\documentclass[11pt,a4paper,twoside]{amsart}

\usepackage{amsfonts, amsmath,amssymb}

\usepackage{hyperref}

\usepackage[all]{xy}

\usepackage[lite]{amsrefs}

\usepackage{amsthm}

\usepackage{geometry}

\usepackage{graphicx}

\usepackage{wrapfig}

\usepackage{subfig}

\usepackage{appendix}

\usepackage{pstricks, pstricks-add}

\usepackage{graphicx}

\newcommand{\aTop}[2]{\begin{array}{c}{#1}\\{#2}\end{array}}

\newcommand{\Hy}{\mathcal{H}}
\newcommand{\ringO}{\mathcal{O}}

\newcommand{\C}{{\mathbb{C}}}
\newcommand{\R}{{\mathbb{R}}}
\newcommand{\Z}{{\mathbb{Z}}}
\newcommand{\F}{{\mathbb{F}}}

\newcommand{\N}{{\mathbb{N}}}
\newcommand{\rationals}{{\mathbb{Q}}}

\newcommand*{\Homol}{\operatorname{H}}
\newcommand{\Farrell}{\widehat{\Homol}}

\newcommand{\PSL}{\mathrm{PSL}}
\newcommand{\PSLO}{\mathrm{PSL}_2\left(\mathcal{O}_{-m}\right)}
\newcommand{\SLO}{\mathrm{SL}_2\left(\mathcal{O}_{-m}\right)}

\renewcommand{\leq}{\leqslant}
\renewcommand{\geq}{\geqslant}

\newcommand{\Afour}{\mathcal{A}_4}
\newcommand{\Afive}{\mathcal{A}_5}
\newcommand{\Sthree}{\mathcal{D}_3}
\newcommand{\Kleinfourgroup}{\mathcal{D}_2}
\newcommand{\Dthree}{\mathcal{D}_3}
\newcommand{\Dfour}{\mathcal{D}_4}
\newcommand{\Dfive}{\mathcal{D}_5}
\newcommand{\Dsix}{\mathcal{D}_6}
\newcommand{\Dl}{\mathcal{D}_\ell}
\newcommand{\Dn}{\mathcal{D}_n}
\newcommand{\Sfour}{\mathcal{S}_4}
\newcommand{\Sfive}{\mathcal{S}_5}
\newcommand{\Stwo}{\mathcal{S}_2}
\newcommand{\Icos}{\mathrm{Icos}_{120}}

\newcommand{\cellCondition}{A}
\newcommand{\IsomorphismConditionSymbol}{B}
\newcommand{\technicalCondition}{B'}

\theoremstyle{plain}
\newtheorem{thm}{\bfseries Theorem}
\newtheorem{theorem}[thm]{\bfseries Theorem}
\newtheorem{Lem}[thm]{\bfseries Lemma}
\newtheorem{lemma}[thm]{\bfseries Lemma}

\newtheorem{proposition}[thm]{\bfseries Proposition}

\newtheorem{corollary}[thm]{\bfseries Corollary}

\newtheorem{definition}[thm]{\bfseries Definition}

\theoremstyle{remark}
\newtheorem{df}[thm]{\bfseries Definition}

\newtheorem{observation}[thm]{\bfseries Observation}

\newtheorem{remark}[thm]{\bfseries Remark}

\newtheorem*{ConditionA}{\bfseries Condition A}
\newtheorem*{isomorphismCondition}{\bfseries Condition B}
\newtheorem*{ConditionBprime}{\bfseries Condition B'}

\newcommand{\circlegraph}{ %reduced 2-torsion subcomplex quotient for m = 7, 15, 39
\begin{pspicture}     (-0.21,-0.155)(0.21,0.155) 
                      \pscircle(0,0.0){0.15}
                      \psdots(0.15,0.0)
\end{pspicture} }

\newcommand{\edgegraph}{ %reduced 2-torsion subcomplex quotient for m = 11, 19, 43, 67, 139, 163
\begin{pspicture}(-0.3,-0.1)(0.3,0.3)
      \psdots(-0.2,0.0)
      \psdots(0.2,0.0)
      \psline(-0.2,0.0)(0.2,0.0)
\end{pspicture} }

\newcommand{\graphFive}{  % reduced 2-torsion subcomplex quotient for m = 5, 10, 13
\begin{pspicture}(-0.21,-0.155)(0.21,0.155)
      \pscircle(0,0.0){0.15}
      \psdots(-0.15,0)
      \psdots(0.15,0)
      \psline(-0.15,0)(0.15,0)
\end{pspicture} }

\newcommand{\graphTwo}{  % reduced 2-torsion subcomplex quotient for m = 2
\begin{pspicture}(0,0.05)(0.8,0.4)
      \pscircle(0.2,0.2){0.15}
      \psdots(0.35,0.2)
      \psline(0.35,0.2)(0.65,0.2)
      \psdots(0.65,0.2)
\end{pspicture} }

\newcommand{\ttsOne}{ 
\begin{pspicture}(-1.6,-0.1)(1.1,0.3)
	\uput{0}[180](0.0,0.2){\footnotesize $(\Z/2)^3 \rtimes \Sthree$}
        \psdots(-0.2,0.0)
        \psdots(0.4,0.0)
        \psline(-0.2,0.0)(0.2,0.0)
        \psline(0.2,0.0)(0.6,0.0)
	\uput{0}[0](0.3,0.2){\footnotesize $\Dthree$}
        \psdots(1.0,0.0)
        \psline(0.6,0.0)(1.0,0.0)
	\uput{0}[0](1.0,0.2){\footnotesize $\Sthree \times \Z/2$}
\end{pspicture} } 

\newcommand{\ttsTwo}{ 
\begin{pspicture}(-1.6,-0.1)(2.6,0.3)
	\uput{0}[180](0.0,0.2){\footnotesize $(\Z/2)^3 \rtimes \Sthree$}
        \psdots(-0.2,0.0)
        \psdots(0.4,0.0)
        \psline(-0.2,0.0)(0.2,0.0)
        \psline(0.2,0.0)(0.6,0.0)
	\uput{0}[0](0.3,0.2){\footnotesize $\Dthree$}
        \psdots(1.0,0.0)
        \psline(0.6,0.0)(1.0,0.0)
	\uput{0}[0](1.0,0.2){\footnotesize $(\Z/2)^3 \rtimes \Sthree$}
\end{pspicture} } 

\newcommand{\ttsThree}{ 
\begin{pspicture}(-0.6,-0.1)(3.2,0.99)
	\uput{0.1}[270](-0.0,0.0){\footnotesize $\Sfour$}
        \psdots(-0.0,0.0)
        \psdots(0.6,0.0)
        \psline(-0.0,0.0)(0.4,0.0)
        \psline(0.4,0.0)(0.8,0.0)
	\uput{0.1}[270](0.6,0.0){\footnotesize $\Dthree$}
        \psdots(1.2,0.0)
        \psline(0.8,0.0)(1.2,0.0)
	\uput{0.1}[0](1.2,0.0){\footnotesize $(\Z/2)^3 \rtimes \Sthree$}
        \psdots(0.25,0.433)
	\uput{0.15}[180](0.25,0.433){\footnotesize $\Dthree$}
        \psdots(0.5,0.866)
	\uput{0.1}[0](0.5,0.866){\footnotesize $(\Z/2)^3 \rtimes \Sthree$}
        \psline(-0.0,0.0)(0.5,0.866)
\end{pspicture} } 

\newcommand{\ttsSeven}{ 
\begin{pspicture}(-0.6,-0.1)(3.2,0.99)
        \psdots(0.6,0.0)
	\uput{0.1}[180](0.6,0.0){\footnotesize $\Dthree$}
        \psdots(1.2,0.0)
        \psline(0.6,0.0)(1.2,0.0)
	\uput{0.1}[0](1.2,0.0){\footnotesize $(\Z/2)^3 \rtimes \Sthree$}
        \psdots(0.6,0.5)
	\uput{0.1}[180](0.6,0.5){\footnotesize $\Dsix$}
        \psdots(1.2,0.5)
        \psline(0.6,0.5)(1.2,0.5)
	\uput{0.1}[0](1.2,0.5){\footnotesize $\Dsix \times \Z/2$}
\end{pspicture} } 

\newcommand{\ttsEight}{ 
\begin{pspicture}(-0.6,-0.3)(3.2,0.99)
	\uput{0.1}[180](0.0,0.0){\footnotesize $\Dthree$}
        \psdots(0.0,0.0)
        \psline(0.0,0.0)(0.6,0.0)
        \psdots(0.6,0.0)
	\uput{0.1}[270](0.6,0.0){\footnotesize $\Sfour$}
        \psline(0.6,0.0)(1.2,0.0)
        \psdots(1.2,0.0)
	\uput{0.1}[270](1.2,0.0){\footnotesize $\Dthree$}
        \psline(1.2,0.0)(2.4,0.0)
        \psdots(1.8,0.0)
	\uput{0.1}[270](1.8,0.0){\footnotesize $\Sfour$}
        \psdots(2.4,0.0)
	\uput{0.1}[0](2.4,0.0){\footnotesize $\Dthree$}
        \psdots(0.6,0.5)
	\uput{0.1}[180](0.6,0.5){\footnotesize $\Dthree$}
        \psdots(1.2,0.5)
        \psline(0.6,0.5)(1.2,0.5)
	\uput{0.1}[0](1.2,0.5){\footnotesize $\Dthree \times \Z/2$}
\end{pspicture} } 

\newcommand{\ttsNine}{ 
\begin{pspicture}(-0.6,-0.3)(3.2,0.6)
        \psdots(0.6,-0.5)
	\uput{0.1}[180](0.6,-0.5){\footnotesize $\Dthree$}
        \psdots(1.2,-0.5)
        \psline(0.6,-0.5)(1.2,-0.5)
	\uput{0.1}[0](1.2,-0.5){\footnotesize $(\Z/2)^3 \rtimes \Sthree$}
        \psdots(0.6,0.0)
	\uput{0.1}[180](0.6,0.0){\footnotesize $\Dthree$}
        \psdots(1.2,0.0)
        \psline(0.6,0.0)(1.2,0.0)
	\uput{0.1}[0](1.2,0.0){\footnotesize $(\Z/2)^3 \rtimes \Sthree$}
        \psdots(0.6,0.5)
	\uput{0.1}[180](0.6,0.5){\footnotesize $\Dthree$}
        \psdots(1.2,0.5)
        \psline(0.6,0.5)(1.2,0.5)
	\uput{0.1}[0](1.2,0.5){\footnotesize $\Dthree \times \Z/2$}
\end{pspicture} } 

\newcommand{\ttsTen}{ 
\begin{pspicture}(-0.6,-0.1)(3.2,0.99)
	\uput{0.1}[270](-0.0,0.0){\footnotesize $\Sfour$}
        \psdots(-0.0,0.0)
        \psdots(0.6,0.0)
        \psline(-0.0,0.0)(0.4,0.0)
        \psline(0.4,0.0)(0.8,0.0)
	\uput{0.1}[270](0.6,0.0){\footnotesize $\Dthree$}
        \psdots(1.2,0.0)
        \psline(0.8,0.0)(1.2,0.0)
	\uput{0.1}[0](1.2,0.0){\footnotesize $\Dthree \times \Z/2$}
        \psdots(0.25,0.433)
	\uput{0.15}[180](0.25,0.433){\footnotesize $\Dthree$}
        \psline(-0.0,0.0)(0.25,0.433)
        \psdots(1.2,0.5)
	\uput{0.1}[180](1.2,0.5){\footnotesize $\Dsix$}
        \psdots(1.8,0.5)
        \psline(1.2,0.5)(1.8,0.5)
	\uput{0.1}[0](1.8,0.5){\footnotesize $\Dsix \times \Z/2$}
\end{pspicture} } 

\newcommand{\ttsEleven}{ 
\begin{pspicture}(-1.2,-0.4)(2.8,0.3)
	\uput{0.2}[180](-0.6,0.0){\footnotesize $\Dthree$}
        \psdots(-0.6,0.0)
        \psdots(0.4,0.0)
        \psline(-0.6,0.0)(0.2,0.0)
        \psline(0.2,0.0)(0.6,0.0)
	\uput{0}[0](-0.2,-0.3){\footnotesize $\Dthree\times \Z/2$}
        \psdots(1.4,0.0)
        \psline(0.6,0.0)(1.4,0.0)
	\uput{0.2}[0](1.4,0.0){\footnotesize $\Dthree$}
        \psdots(2.5,0.0)
	\uput{0.2}[0](2.5,0.0){\footnotesize $\Dsix$}
\end{pspicture} } 

\newcommand{\ttsThirteen}{ 
\begin{pspicture}(-0.4,0.3)(2.8,1.2)
	\psdots(0.6,1.0)
	\uput{0.1}[0](0.6,1.0){\footnotesize $\Dsix$}
	\psdots(1.4,1.0)
	\uput{0.1}[0](1.4,1.0){\footnotesize $\Dthree$}
	\psdots(2.2,1.0)
	\uput{0.1}[0](2.2,1.0){\footnotesize $\Dthree$}
        \psdots(0.6,0.5)
	\uput{0.1}[180](0.6,0.5){\footnotesize $\Dthree$}
        \psdots(1.2,0.5)
        \psline(0.6,0.5)(1.2,0.5)
	\uput{0.1}[0](1.2,0.5){\footnotesize $\Dthree \times \Z/2$}
\end{pspicture} } 

\newcommand{\ttsFourteen}{ 
\begin{pspicture}(-0.6,-0.1)(3.2,0.99)
        \psdots(0.6,0.0)
	\uput{0.1}[180](0.6,0.0){\footnotesize $\Dsix$}
        \psdots(1.2,0.0)
        \psline(0.6,0.0)(1.2,0.0)
	\uput{0.1}[0](1.2,0.0){\footnotesize $\Dsix \times \Z/2$}
        \psdots(0.6,0.5)
	\uput{0.1}[180](0.6,0.5){\footnotesize $\Dsix$}
        \psdots(1.2,0.5)
        \psline(0.6,0.5)(1.2,0.5)
	\uput{0.1}[0](1.2,0.5){\footnotesize $\Dsix \times \Z/2$}
	\psdots(2.99,0.25)
	\uput{0.1}[0](2.99,0.25){\footnotesize $\Dthree$}
\end{pspicture} } 

\newcommand{\ttsFifteen}{ 
\begin{pspicture}(0.1,-0.1)(3.2,0.99)
        \psdots(0.6,0.0)
	\uput{0.1}[180](0.6,0.0){\footnotesize $\Dthree$}
        \psdots(1.2,0.0)
        \psline(0.6,0.0)(1.2,0.0)
	\uput{0.1}[0](1.2,0.0){\footnotesize $\Dthree \times \Z/2$}
        \psdots(0.6,0.5)
	\uput{0.1}[180](0.6,0.5){\footnotesize $\Dthree$}
        \psdots(1.2,0.5)
        \psline(0.6,0.5)(1.2,0.5)
	\uput{0.1}[0](1.2,0.5){\footnotesize $\Dthree \times \Z/2$}
	\psdots(2.99,0.25)
	\uput{0.1}[0](2.99,0.25){\footnotesize $\Dsix$}
\end{pspicture} } 

\newcommand{\ttsSixteen}{ 
\begin{pspicture}(-0.6,-0.3)(3.2,0.99)
	\uput{0.1}[180](0.0,0.0){\footnotesize $\Dthree$}
        \psdots(0.0,0.0)
        \psline(0.0,0.0)(0.6,0.0)
        \psdots(0.6,0.0)
	\uput{0.1}[270](0.6,0.0){\footnotesize $\Sfour$}
        \psline(0.6,0.0)(1.2,0.0)
        \psdots(1.2,0.0)
	\uput{0.1}[0](1.2,0.0){\footnotesize $\Dthree$}
	\psdots(0,0.5)
        \psline(0,0.5)(0.6,0.5)
	\uput{0.1}[180](0,0.5){\footnotesize $\Dthree$}
        \psdots(0.6,0.5)
	\uput{0.1}[90](0.6,0.5){\footnotesize $\Sfour$}
        \psdots(1.2,0.5)
        \psline(0.6,0.5)(1.2,0.5)
	\uput{0.1}[0](1.2,0.5){\footnotesize $\Dthree$}
	\psdots(2.2,0.25)
	\uput{0.1}[0](2.2,0.25){\footnotesize $\Dthree$}
\end{pspicture} } 

\newcommand{\ttsEighteen}{ 
\begin{pspicture}(-1.6,-0.1)(2.6,0.8)
	\uput{0}[180](0.0,0.2){\footnotesize $(\Z/2)^3 \rtimes \Sthree$}
        \psdots(-0.2,0.0)
        \psdots(0.4,0.0)
        \psline(-0.2,0.0)(0.2,0.0)
        \psline(0.2,0.0)(0.6,0.0)
	\uput{0}[0](0.3,0.2){\footnotesize $\Dthree$}
        \psdots(1.0,0.0)
        \psline(0.6,0.0)(1.0,0.0)
	\uput{0}[0](1.0,0.2){\footnotesize $(\Z/2)^3 \rtimes \Sthree$}
	\uput{0}[180](0.0,0.7){\footnotesize $(\Z/2)^3 \rtimes \Sthree$}
        \psdots(-0.2,0.5)
        \psdots(0.4,0.5)
        \psline(-0.2,0.5)(0.2,0.5)
        \psline(0.2,0.5)(0.6,0.5)
	\uput{0}[0](0.3,0.7){\footnotesize $\Dthree$}
        \psdots(1.0,0.5)
        \psline(0.6,0.5)(1.0,0.5)
	\uput{0}[0](1.0,0.7){\footnotesize $(\Z/2)^3 \rtimes \Sthree$}
\end{pspicture} } 

\newcommand{\ttsNineteen}{ 
\begin{pspicture}(-1.6,-0.1)(2.6,0.3)
	\uput{0}[180](0.0,0.2){\footnotesize $(\Z/2)^3 \rtimes \Sthree$}
        \psdots(-0.2,0.0)
        \psdots(0.4,0.0)
        \psline(-0.2,0.0)(0.2,0.0)
        \psline(0.2,0.0)(0.6,0.0)
	\uput{0}[0](0.3,0.2){\footnotesize $\Dthree$}
        \psdots(1.0,0.0)
        \psline(0.6,0.0)(1.0,0.0)
	\uput{0}[0](1.0,0.2){\footnotesize $\Icos$}
\end{pspicture} } 

\newcommand{\ttsTwenty}{ 
\begin{pspicture}(-0.6,-0.1)(3.2,0.99)
	\uput{0.1}[270](-0.0,0.0){\footnotesize $\Sfour$}
        \psdots(-0.0,0.0)
        \psdots(0.6,0.0)
        \psline(-0.0,0.0)(0.4,0.0)
        \psline(0.4,0.0)(0.8,0.0)
	\uput{0.1}[270](0.6,0.0){\footnotesize $\Dthree$}
        \psdots(1.2,0.0)
        \psline(0.8,0.0)(1.2,0.0)
	\uput{0.1}[0](1.2,0.0){\footnotesize $\Icos$}
        \psdots(0.25,0.433)
	\uput{0.15}[180](0.25,0.433){\footnotesize $\Dthree$}
        \psdots(0.5,0.866)
	\uput{0.1}[0](0.5,0.866){\footnotesize $\Icos$}
        \psline(-0.0,0.0)(0.5,0.866)
\end{pspicture} }

\newcommand{\ttsTwentyone}{ 
\begin{pspicture}(-1.2,-0.3)(2.2,0.99)
	\uput{0.1}[180](0.0,0.0){\footnotesize $\Icos$}
        \psdots(0.0,0.0)
        \psline(0.0,0.0)(0.6,0.0)
        \psdots(0.6,0.0)
	\uput{0.1}[270](0.6,0.0){\footnotesize $\Dthree$}
        \psline(0.6,0.0)(1.2,0.0)
        \psdots(1.2,0.0)
	\uput{0.1}[0](1.2,0.0){\footnotesize $\Icos$}
	\psdots(0,0.5)
        \psline(0,0.5)(0.6,0.5)
	\uput{0.1}[180](0,0.5){\footnotesize $\Icos$}
        \psdots(0.6,0.5)
	\uput{0.1}[90](0.6,0.5){\footnotesize $\Dthree$}
        \psdots(1.2,0.5)
        \psline(0.6,0.5)(1.2,0.5)
	\uput{0.1}[0](1.2,0.5){\footnotesize $\Icos$}
\end{pspicture} } 

\newcommand{\ttsTwentytwo}{ 
\begin{pspicture}(-1.2,-0.3)(2.2,0.99)
	\uput{0.1}[180](0.0,0.0){\footnotesize $\Icos$}
        \psdots(0.0,0.0)
        \psline(0.0,0.0)(0.6,0.0)
        \psdots(0.6,0.0)
	\uput{0.1}[270](0.6,0.0){\footnotesize $\Dthree$}
        \psline(0.6,0.0)(1.2,0.0)
        \psdots(1.2,0.0)
	\uput{0.1}[0](1.2,0.0){\footnotesize $\Dthree \times \Z/2$}
	\psdots(0,0.5)
        \psline(0,0.5)(0.6,0.5)
	\uput{0.1}[180](0,0.5){\footnotesize $\Icos$}
        \psdots(0.6,0.5)
	\uput{0.1}[90](0.6,0.5){\footnotesize $\Dthree$}
        \psdots(1.2,0.5)
        \psline(0.6,0.5)(1.2,0.5)
	\uput{0.1}[0](1.2,0.5){\footnotesize $\Dthree \times \Z/2$}
\end{pspicture} } 

\newcommand{\ttsTwentythree}{ 
\begin{pspicture}(-0.6,-0.1)(3.2,0.99)
	\uput{0.2}[180](-0.0,0.0){\footnotesize $\Sfour$}
        \psdots(-0.0,0.0)
        \psdots(0.6,0.0)
        \psline(-0.0,0.0)(0.4,0.0)
        \psline(0.4,0.0)(0.8,0.0)
	\uput{0.1}[270](0.6,0.0){\footnotesize $\Dthree$}
        \psdots(1.2,0.0)
        \psline(0.8,0.0)(1.2,0.0)
	\uput{0.1}[0](1.2,0.0){\footnotesize $\Icos$}
        \psdots(0.25,0.433)
	\uput{0.15}[180](0.25,0.433){\footnotesize $\Dthree$}
        \psdots(0.5,0.866)
        \psline(-0.0,0.0)(0.5,0.866)
	\uput{0.2}[180](0.5,0.866){\footnotesize $\Sfour$}
        \psline(0.5,0.866)(1.1,0.866)
        \psdots(1.1,0.866)
	\uput{0.1}[270](1.1,0.866){\footnotesize $\Dthree$}
        \psline(1.7,0.866)(1.1,0.866)
        \psdots(1.7,0.866)
	\uput{0.1}[0](1.7,0.866){\footnotesize $\Icos$}
\end{pspicture} } 

\newcommand{\ttsTwentyfour}{ 
\begin{pspicture}(-1.2,-0.3)(.2,0.39)
	\uput{0.1}[180](0.0,0.0){\footnotesize $\Icos$}
        \psdots(0.0,0.0)
        \psline(0.0,0.0)(0.6,0.0)
        \psdots(0.6,0.0)
	\uput{0.1}[270](0.6,0.0){\footnotesize $\Dthree$}
        \psline(0.6,0.0)(1.2,0.0)
        \psdots(1.2,0.0)
	\uput{0.1}[0](1.2,0.0){\footnotesize $\Icos$}
\end{pspicture} } 

\newcommand{\ttsTwentyfive}{ 
\begin{pspicture}(-0.6,-0.1)(3.2,0.99)
	\uput{0.2}[180](-0.0,0.0){\footnotesize $\Sfour$}
        \psdots(-0.0,0.0)
        \psdots(0.6,0.0)
        \psline(-0.0,0.0)(0.4,0.0)
        \psline(0.4,0.0)(0.8,0.0)
	\uput{0.1}[270](0.6,0.0){\footnotesize $\Dthree$}
        \psdots(1.2,0.0)
        \psline(0.8,0.0)(1.2,0.0)
	\uput{0.1}[0](1.2,0.0){\footnotesize $(\Z/2)^3 \rtimes \Sthree$}
        \psdots(0.25,0.433)
	\uput{0.15}[180](0.25,0.433){\footnotesize $\Dthree$}
        \psdots(0.5,0.866)
        \psline(-0.0,0.0)(0.5,0.866)
	\uput{0.2}[180](0.5,0.866){\footnotesize $\Sfour$}
        \psline(0.5,0.866)(1.1,0.866)
        \psdots(1.1,0.866)
	\uput{0.1}[270](1.1,0.866){\footnotesize $\Dthree$}
        \psline(1.7,0.866)(1.1,0.866)
        \psdots(1.7,0.866)
	\uput{0.1}[0](1.7,0.866){\footnotesize $(\Z/2)^3 \rtimes \Sthree$}
\end{pspicture} } 

\newcommand{\ttsTwentysix}{ 
\begin{pspicture}(-1.2,-0.3)(2.2,0.99)
	\uput{0.1}[180](0.0,0.0){\footnotesize $\Icos$}
        \psdots(0.0,0.0)
        \psline(0.0,0.0)(0.6,0.0)
        \psdots(0.6,0.0)
	\uput{0.1}[270](0.6,0.0){\footnotesize $\Dthree$}
        \psline(0.6,0.0)(1.2,0.0)
        \psdots(1.2,0.0)
	\uput{0.1}[0](1.2,0.0){\footnotesize $(\Z/2)^3 \rtimes \Sthree$}
	\psdots(0,0.5)
        \psline(0,0.5)(0.6,0.5)
	\uput{0.1}[180](0,0.5){\footnotesize $\Icos$}
        \psdots(0.6,0.5)
	\uput{0.1}[90](0.6,0.5){\footnotesize $\Dthree$}
        \psdots(1.2,0.5)
        \psline(0.6,0.5)(1.2,0.5)
	\uput{0.1}[0](1.2,0.5){\footnotesize $(\Z/2)^3 \rtimes \Sthree$}
\end{pspicture} }

\newcommand{\ttsTwentyseven}{ 
\begin{pspicture}(-0.6,-0.3)(3.2,0.6)
        \psdots(0.6,-0.5)
	\uput{0.1}[180](0.6,-0.5){\footnotesize $\Dthree$}
        \psdots(1.2,-0.5)
        \psline(0.6,-0.5)(1.2,-0.5)
	\uput{0.1}[0](1.2,-0.5){\footnotesize $\Icos$}
        \psdots(0.6,0.0)
	\uput{0.1}[180](0.6,0.0){\footnotesize $\Dthree$}
        \psdots(1.2,0.0)
        \psline(0.6,0.0)(1.2,0.0)
	\uput{0.1}[0](1.2,0.0){\footnotesize $\Icos$}
        \psdots(0.6,0.5)
	\uput{0.1}[180](0.6,0.5){\footnotesize $\Dthree$}
        \psdots(1.2,0.5)
        \psline(0.6,0.5)(1.2,0.5)
	\uput{0.1}[0](1.2,0.5){\footnotesize $\Dthree \times \Z/2$}
\end{pspicture} }

\newcommand{\ttsTwentyeight}{ 
\begin{pspicture}(-0.6,-0.1)(3.2,0.99)
        \psdots(0.6,0.0)
	\uput{0.1}[180](0.6,0.0){\footnotesize $\Dthree$}
        \psdots(1.2,0.0)
        \psline(0.6,0.0)(1.2,0.0)
	\uput{0.1}[0](1.2,0.0){\footnotesize $\Icos$}
        \psdots(0.6,0.5)
	\uput{0.1}[180](0.6,0.5){\footnotesize $\Dsix$}
        \psdots(1.2,0.5)
        \psline(0.6,0.5)(1.2,0.5)
	\uput{0.1}[0](1.2,0.5){\footnotesize $\Dsix \times \Z/2$}
\end{pspicture} } 

\newcommand{\ttsTwentynine}{ 
\begin{pspicture}(0.1,-0.1)(3.2,0.99)
        \psdots(0.6,0.0)
	\uput{0.1}[180](0.6,0.0){\footnotesize $\Dthree$}
        \psdots(1.2,0.0)
        \psline(0.6,0.0)(1.2,0.0)
	\uput{0.1}[0](1.2,0.0){\footnotesize $\Icos$}
        \psdots(0.6,0.5)
	\uput{0.1}[180](0.6,0.5){\footnotesize $\Dthree$}
        \psdots(1.2,0.5)
        \psline(0.6,0.5)(1.2,0.5)
	\uput{0.1}[0](1.2,0.5){\footnotesize $\Icos$}
	\psdots(2.99,0.25)
	\uput{0.1}[0](2.99,0.25){\footnotesize $\Dsix$}
\end{pspicture} } 

\newcommand{\ttsThirty}{ 
\begin{pspicture}(0.1,-0.1)(3.5,0.99)
        \psdots(0.6,0.0)
	\uput{0.1}[180](0.6,0.0){\footnotesize $\Dthree$}
        \psdots(1.2,0.0)
        \psline(0.6,0.0)(1.2,0.0)
	\uput{0.1}[0](1.2,0.0){\footnotesize $(\Z/2)^3 \rtimes \Sthree$}
        \psdots(0.6,0.5)
	\uput{0.1}[180](0.6,0.5){\footnotesize $\Dthree$}
        \psdots(1.2,0.5)
        \psline(0.6,0.5)(1.2,0.5)
	\uput{0.1}[0](1.2,0.5){\footnotesize $(\Z/2)^3 \rtimes \Sthree$}
	\psdots(3.2,0.25)
	\uput{0.1}[0](3.2,0.25){\footnotesize $\Dsix$}
\end{pspicture} } 

\newcommand{\ttsThirtyone}{ 
\begin{pspicture}(-1.6,-0.1)(2.6,0.3)
	\uput{0}[180](0.0,0.2){\footnotesize $(\Z/2)^3 \rtimes \Sthree$}
        \psdots(-0.2,0.0)
        \psdots(0.4,0.0)
        \psline(-0.2,0.0)(0.2,0.0)
        \psline(0.2,0.0)(0.6,0.0)
	\uput{0}[0](0.3,0.2){\footnotesize $\Dthree$}
        \psdots(1.0,0.0)
        \psline(0.6,0.0)(1.0,0.0)
	\uput{0}[0](1.0,0.2){\footnotesize $(\Z/2)^3 \rtimes \Sthree$}
\end{pspicture} } 

\newcommand{\ttsThirtytwo}{ 
\begin{pspicture}(-0.6,-0.3)(3.2,0.99)
	\uput{0.1}[180](0.0,0.0){\footnotesize $\Dthree$}
        \psdots(0.0,0.0)
        \psline(0.0,0.0)(0.6,0.0)
        \psdots(0.6,0.0)
	\uput{0.1}[270](0.6,0.0){\footnotesize $\Sfour$}
        \psline(0.6,0.0)(1.2,0.0)
        \psdots(1.2,0.0)
	\uput{0.1}[270](1.2,0.0){\footnotesize $\Dthree$}
        \psline(1.2,0.0)(2.4,0.0)
        \psdots(1.8,0.0)
	\uput{0.1}[270](1.8,0.0){\footnotesize $\Sfour$}
        \psdots(2.4,0.0)
	\uput{0.1}[0](2.4,0.0){\footnotesize $\Dthree$}
        \psdots(1.2,0.5)
	\uput{0.1}[0](1.2,0.5){\footnotesize $\Dsix$}
\end{pspicture} }

\label{5-torsion subcomplexes}

\newcommand{\ftsNineteen}{ 
\begin{pspicture}(-1.2,-0.3)(.2,0.39)
	\uput{0.1}[180](0.0,0.0){\footnotesize $\Icos$}
        \psdots(0.0,0.0)
        \psline(0.0,0.0)(0.6,0.0)
        \psdots(0.6,0.0)
	\uput{0.1}[90](0.6,0.0){\footnotesize $\Dfive$}
        \psline(0.6,0.0)(1.2,0.0)
        \psdots(1.2,0.0)
	\uput{0.1}[0](1.2,0.0){\footnotesize $\Dfive \times \Z/2$}
\end{pspicture} } 

\newcommand{\ftsTwenty}{ 
\begin{pspicture}(-1.2,-0.3)(.2,0.39)
	\uput{0.1}[180](0.0,0.0){\footnotesize $\Icos$}
        \psdots(0.0,0.0)
        \psline(0.0,0.0)(0.6,0.0)
        \psdots(0.6,0.0)
	\uput{0.1}[90](0.6,0.0){\footnotesize $\Dfive$}
        \psline(0.6,0.0)(1.2,0.0)
        \psdots(1.2,0.0)
	\uput{0.1}[0](1.2,0.0){\footnotesize $\Icos$}
\end{pspicture} } 

\newcommand{\ftsTwentyone}{ 
\begin{pspicture}(-1.2,-0.3)(2.2,0.99)
	\uput{0.1}[180](0.0,0.0){\footnotesize $\Icos$}
        \psdots(0.0,0.0)
        \psline(0.0,0.0)(0.6,0.0)
        \psdots(0.6,0.0)
	\uput{0.1}[270](0.6,0.0){\footnotesize $\Dfive$}
        \psline(0.6,0.0)(1.2,0.0)
        \psdots(1.2,0.0)
	\uput{0.1}[0](1.2,0.0){\footnotesize $\Icos$}
	\psdots(0,0.5)
        \psline(0,0.5)(0.6,0.5)
	\uput{0.1}[180](0,0.5){\footnotesize $\Icos$}
        \psdots(0.6,0.5)
	\uput{0.1}[90](0.6,0.5){\footnotesize $\Dfive$}
        \psdots(1.2,0.5)
        \psline(0.6,0.5)(1.2,0.5)
	\uput{0.1}[0](1.2,0.5){\footnotesize $\Icos$}
\end{pspicture} } 

\newcommand{\ftsTwentyfour}{ 
\begin{pspicture}(-1.2,-0.3)(2.2,0.99)
	\uput{0.1}[180](0.0,0.0){\footnotesize $\Icos$}
        \psdots(0.0,0.0)
        \psline(0.0,0.0)(0.6,0.0)
        \psdots(0.6,0.0)
	\uput{0.1}[270](0.6,0.0){\footnotesize $\Dfive$}
        \psline(0.6,0.0)(1.2,0.0)
        \psdots(1.2,0.0)
	\uput{0.1}[0](1.2,0.0){\footnotesize $\Dfive \times \Z/2$}
	\psdots(0,0.5)
        \psline(0,0.5)(0.6,0.5)
	\uput{0.1}[180](0,0.5){\footnotesize $\Icos$}
        \psdots(0.6,0.5)
	\uput{0.1}[90](0.6,0.5){\footnotesize $\Dfive$}
        \psdots(1.2,0.5)
        \psline(0.6,0.5)(1.2,0.5)
	\uput{0.1}[0](1.2,0.5){\footnotesize $\Dfive \times \Z/2$}
\end{pspicture} }

\newcommand{\illustration}{ % three adjacent edges, labelled
\begin{pspicture}(-0.4,-0.1)(1.3,0.3)
        \psdots(-0.3,0.0)
        \psline(-0.3,0.0)(0.2,0.0)
        \uput{0}[0](-0.2,0.2){\footnotesize $\beta e$}
        \psdots(0.2,0.0)
        \uput{0}[0](0.1,-0.2){\footnotesize $v$}
        \psline(0.2,0.0)(0.7,0.0)
        \uput{0}[0](0.4,0.2){\footnotesize $e$}
        \psdots(0.7,0.0)
        \psline(0.7,0.0)(1.2,0.0)
        \uput{0}[0](0.8,0.2){\footnotesize $\gamma e$}
        \psdots(1.2,0.0)
        \uput{0}[0](1.1,-0.2){\footnotesize $\gamma v$}
\end{pspicture} }

\newcommand{\CoxeterDiagramOfSfive}{ 
\begin{pspicture}(-0.3,-0.3)(1.1,0.3)
        \psdots(-0.2,0.0)
        \psdots(0.2,0.0)
        \psline(-0.2,0.0)(0.2,0.0)
        \psdots(0.6,0.0)
        \psline(0.2,0.0)(0.6,0.0)
        \psdots(1.0,0.0)
        \psline(0.6,0.0)(1.0,0.0)
\end{pspicture} }

\newcommand{\CoxeterDiagramOfSfour}{ % two adjacent edges
\begin{pspicture}(-0.4,-0.1)(0.8,0.3)
        \psdots(-0.3,0.0)
        \psline(-0.3,0.0)(0.2,0.0)
        \psdots(0.2,0.0)
        \psline(0.2,0.0)(0.7,0.0)
        \psdots(0.7,0.0)
\end{pspicture} }

\newcommand{\CoxeterDiagramDfour}{ 
\begin{pspicture}(-0.3,-0.45)(0.6,0.45)
        \psdots(-0.25,0.433)
        \psdots(-0.25,-0.433)
        \psline(-0.25,0.433)(0.0,0.0)
        \psline(-0.25,-0.433)(0.0,0.0)
        \psdots(0.0,0.0)
        \psdots(0.5,0.0)
        \psline(0.5,0.0)(0.0,0.0)
\end{pspicture} } 

\newcommand{\Bthree}{ % Coxeter diagram B_3
\begin{pspicture}(-0.4,-0.1)(0.8,0.3)
        \psdots(-0.3,0.0)
        \psline(-0.3,0.0)(0.2,0.0)
        \uput{0}[0](-0.1,0.2){\footnotesize $4$}
        \psdots(0.2,0.0)
        \psline(0.2,0.0)(0.7,0.0)
        \psdots(0.7,0.0)
\end{pspicture} }

\newcommand{\Bfour}{ 
\begin{pspicture}(-0.3,-0.3)(1.1,0.3)
        \psdots(-0.2,0.0)
        \psdots(0.2,0.0)
        \psline(-0.2,0.0)(0.2,0.0)
        \psdots(0.6,0.0)
        \psline(0.2,0.0)(0.6,0.0)
	\uput{0}[0](-0.1,0.2){\footnotesize $4$}
        \psdots(1.0,0.0)
        \psline(0.6,0.0)(1.0,0.0)
\end{pspicture} }

\newcommand{\Hthree}{ % Coxeter diagram H_3
\begin{pspicture}(-0.4,-0.1)(0.8,0.3)
        \psdots(-0.3,0.0)
        \psline(-0.3,0.0)(0.2,0.0)
        \uput{0}[0](-0.1,0.2){\footnotesize $5$}
        \psdots(0.2,0.0)
        \psline(0.2,0.0)(0.7,0.0)
        \psdots(0.7,0.0)
\end{pspicture} }

\newcommand{\CTone}{ 
\begin{pspicture}(-0.3,-0.3)(1.1,0.3)
        \psdots(-0.2,0.0)
        \psdots(0.2,0.0)
        \psline(-0.2,0.0)(0.2,0.0)
        \psdots(0.6,0.0)
        \psline(0.2,0.0)(0.6,0.0)
	\uput{0}[0](0.3,0.2){\footnotesize $4$}
        \psdots(1.0,0.0)
        \psline(0.6,0.0)(1.0,0.0)
	\uput{0}[0](0.7,0.2){\footnotesize $4$}
\end{pspicture} } 

\newcommand{\CTtwo}{ 
\begin{pspicture}(-0.3,-0.45)(0.6,0.45)
        \psdots(-0.25,0.433)
        \psdots(-0.25,-0.433)
        \psline(-0.25,0.433)(0.0,0.0)
        \psline(-0.25,-0.433)(0.0,0.0)
        \psdots(0.0,0.0)
        \psdots(0.5,0.0)
        \psline(0.5,0.0)(0.0,0.0)
	\uput{0.04}[45](-0.125,0.2165){\footnotesize $4$}
	\uput{0.04}[315](-0.125,-0.2165){\footnotesize $4$}
\end{pspicture} } 

\newcommand{\CTthree}{ 
\begin{pspicture}(-0.1,-0.1)(0.8,0.8)
        \psdots(0.0,0.0)
        \psdots(0.6,0.0)
        \psdots(0.6,0.6)
        \psdots(0,0.6)
        \psline(0.0,0.0)(0.6,0.0)
        \psline(0.6,0.0)(0.6,0.6)
        \psline(0.6,0.6)(0,0.6)
        \psline(0,0.6)(0.0,0.0)
	\uput{0.1}[270](0.3,0.0){\footnotesize $4$}
	\uput{0.1}[0](0.6,0.3){\footnotesize $4$}
\end{pspicture} } 

\newcommand{\CTseven}{ 
\begin{pspicture}(-0.3,-0.3)(1.1,0.3)
        \psdots(-0.2,0.0)
        \psdots(0.2,0.0)
        \psline(-0.2,0.0)(0.2,0.0)
        \psdots(0.6,0.0)
        \psline(0.2,0.0)(0.6,0.0)
	\uput{0}[0](-0.1,0.2){\footnotesize $4$}
        \psdots(1.0,0.0)
        \psline(0.6,0.0)(1.0,0.0)
	\uput{0}[0](0.7,0.2){\footnotesize $6$}
\end{pspicture} } 

\newcommand{\CTeight}{ 
\begin{pspicture}(-0.3,-0.65)(0.6,0.65)
        \psdots(-0.25,0.433)
        \psdots(-0.25,-0.433)
        \psline(-0.25,0.433)(0.0,0.0)
        \psline(-0.25,-0.433)(0.0,0.0)
        \psdots(0.0,0.0)
        \psdots(0.5,0.0)
        \psline(0.5,0.0)(0.0,0.0)
        \psline(-0.25,0.433)(-0.25,-0.433)
\end{pspicture} } 

\newcommand{\CTnine}{ 
\begin{pspicture}(-0.3,-0.55)(0.6,0.55)
        \psdots(-0.25,0.433)
        \psdots(-0.25,-0.433)
        \psline(-0.25,0.433)(0.0,0.0)
        \psline(-0.25,-0.433)(0.0,0.0)
        \psdots(0.0,0.0)
        \psdots(0.5,0.0)
        \psline(0.5,0.0)(0.0,0.0)
	\uput{0}[0](0.25,0.2){\footnotesize $4$}
        \psline(-0.25,0.433)(-0.25,-0.433)
\end{pspicture} } 

\newcommand{\CTten}{ 
\begin{pspicture}(-0.3,-0.3)(1.1,0.3)
        \psdots(-0.2,0.0)
        \psdots(0.2,0.0)
        \psline(-0.2,0.0)(0.2,0.0)
        \psdots(0.6,0.0)
        \psline(0.2,0.0)(0.6,0.0)
        \psdots(1.0,0.0)
        \psline(0.6,0.0)(1.0,0.0)
	\uput{0}[0](0.7,0.2){\footnotesize $6$}
\end{pspicture} } 

\newcommand{\CTeleven}{ 
\begin{pspicture}(-0.3,-0.45)(0.6,0.45)
        \psdots(-0.25,0.433)
        \psdots(-0.25,-0.433)
        \psline(-0.25,0.433)(0.0,0.0)
        \psline(-0.25,-0.433)(0.0,0.0)
        \psdots(0.0,0.0)
        \psdots(0.5,0.0)
        \psline(0.5,0.0)(0.0,0.0)
	\uput{0.04}[45](-0.125,0.2165){\footnotesize $6$}
\end{pspicture} } 

\newcommand{\CTtwelve}{ 
\begin{pspicture}(-0.3,-0.45)(0.6,0.45)
        \psdots(-0.25,0.433)
        \psline(-0.25,0.433)(0.5,0.0)
        \psline(-0.25,-0.433)(0.5,0.0)
        \psdots(-0.25,-0.433)
        \psline(-0.25,0.433)(0.0,0.0)
        \psline(-0.25,-0.433)(0.0,0.0)
        \psdots(0.0,0.0)
        \psdots(0.5,0.0)
        \psline(0.5,0.0)(0.0,0.0)
        \psline(-0.25,0.433)(-0.25,-0.433)
\end{pspicture} } 

\newcommand{\CTthirteen}{ 
\begin{pspicture}(-0.3,-0.55)(0.6,0.55)
        \psdots(-0.25,0.433)
        \psdots(-0.25,-0.433)
        \psline(-0.25,0.433)(0.0,0.0)
        \psline(-0.25,-0.433)(0.0,0.0)
        \psdots(0.0,0.0)
        \psdots(0.5,0.0)
        \psline(0.5,0.0)(0.0,0.0)
	\uput{0}[0](0.25,0.2){\footnotesize $6$}
        \psline(-0.25,0.433)(-0.25,-0.433)
\end{pspicture} } 

\newcommand{\CTfourteen}{ 
\begin{pspicture}(-0.3,-0.3)(1.1,0.3)
        \psdots(-0.2,0.0)
        \psdots(0.2,0.0)
        \psline(-0.2,0.0)(0.2,0.0)
        \psdots(0.6,0.0)
        \psline(0.2,0.0)(0.6,0.0)
	\uput{0}[0](-0.1,0.2){\footnotesize $6$}
        \psdots(1.0,0.0)
        \psline(0.6,0.0)(1.0,0.0)
	\uput{0}[0](0.7,0.2){\footnotesize $6$}
\end{pspicture} } 

\newcommand{\CTfifteen}{ 
\begin{pspicture}(-0.3,-0.3)(1.1,0.3)
        \psdots(-0.2,0.0)
        \psdots(0.2,0.0)
        \psline(-0.2,0.0)(0.2,0.0)
        \psdots(0.6,0.0)
        \psline(0.2,0.0)(0.6,0.0)
        \psdots(1.0,0.0)
        \psline(0.6,0.0)(1.0,0.0)
	\uput{0}[0](0.3,0.2){\footnotesize $6$}
\end{pspicture} }

\newcommand{\CTsixteen}{ 
\begin{pspicture}(-0.3,-0.45)(0.6,0.45)
        \psdots(-0.25,0.433)
        \psline(-0.25,0.433)(0.5,0.0)
        \psline(-0.25,-0.433)(0.5,0.0)
        \psdots(-0.25,-0.433)
        \psline(-0.25,0.433)(0.0,0.0)
        \psline(-0.25,-0.433)(0.0,0.0)
        \psdots(0.0,0.0)
        \psdots(0.5,0.0)
        \psline(0.5,0.0)(0.0,0.0)
\end{pspicture} } 

\newcommand{\CTseventeen}{ 
\begin{pspicture}(-0.1,-0.1)(0.8,0.8)
        \psdots(0.0,0.0)
        \psdots(0.6,0.0)
        \psdots(0.6,0.6)
        \psdots(0,0.6)
        \psline(0.0,0.0)(0.6,0.0)
        \psline(0.6,0.0)(0.6,0.6)
        \psline(0.6,0.6)(0,0.6)
        \psline(0,0.6)(0.0,0.0)
	\uput{0.1}[180](0.0,0.3){\footnotesize $6$}
	\uput{0.1}[0](0.6,0.3){\footnotesize $6$}
\end{pspicture} } 

\newcommand{\CTeighteen}{ 
\begin{pspicture}(-0.1,-0.1)(0.8,0.8)
        \psdots(0.0,0.0)
        \psdots(0.6,0.0)
        \psdots(0.6,0.6)
        \psdots(0,0.6)
        \psline(0.0,0.0)(0.6,0.0)
        \psline(0.6,0.0)(0.6,0.6)
        \psline(0.6,0.6)(0,0.6)
        \psline(0,0.6)(0.0,0.0)
	\uput{0.1}[180](0.0,0.3){\footnotesize $4$}
	\uput{0.1}[0](0.6,0.3){\footnotesize $4$}
\end{pspicture} } 

\newcommand{\CTnineteen}{ 
\begin{pspicture}(-0.3,-0.3)(1.1,0.3)
        \psdots(-0.2,0.0)
        \psdots(0.2,0.0)
        \psline(-0.2,0.0)(0.2,0.0)
        \psdots(0.6,0.0)
        \psline(0.2,0.0)(0.6,0.0)
	\uput{0}[0](-0.1,0.2){\footnotesize $4$}
        \psdots(1.0,0.0)
        \psline(0.6,0.0)(1.0,0.0)
	\uput{0}[0](0.7,0.2){\footnotesize $5$}
\end{pspicture} } 

\newcommand{\CTtwenty}{ 
\begin{pspicture}(-0.3,-0.45)(0.6,0.45)
        \psdots(-0.25,0.433)
        \psdots(-0.25,-0.433)
        \psline(-0.25,0.433)(0.0,0.0)
        \psline(-0.25,-0.433)(0.0,0.0)
        \psdots(0.0,0.0)
        \psdots(0.5,0.0)
        \psline(0.5,0.0)(0.0,0.0)
	\uput{0.04}[45](-0.125,0.2165){\footnotesize $5$}
\end{pspicture} } 

\newcommand{\CTtwentyone}{ 
\begin{pspicture}(-0.1,-0.1)(0.8,0.8)
        \psdots(0.0,0.0)
        \psdots(0.6,0.0)
        \psdots(0.6,0.6)
        \psdots(0,0.6)
        \psline(0.0,0.0)(0.6,0.0)
        \psline(0.6,0.0)(0.6,0.6)
        \psline(0.6,0.6)(0,0.6)
        \psline(0,0.6)(0.0,0.0)
	\uput{0.1}[180](0.0,0.3){\footnotesize $5$}
	\uput{0.1}[0](0.6,0.3){\footnotesize $5$}
\end{pspicture} } 

\newcommand{\CTtwentytwo}{ 
\begin{pspicture}(-0.3,-0.3)(1.1,0.3)
        \psdots(-0.2,0.0)
        \psdots(0.2,0.0)
        \psline(-0.2,0.0)(0.2,0.0)
        \psdots(0.6,0.0)
        \psline(0.2,0.0)(0.6,0.0)
        \psdots(1.0,0.0)
        \psline(0.6,0.0)(1.0,0.0)
	\uput{0}[0](0.3,0.2){\footnotesize $5$}
\end{pspicture} }

\newcommand{\CTtwentythree}{ 
\begin{pspicture}(-0.1,-0.1)(0.8,0.8)
        \psdots(0.0,0.0)
        \psdots(0.6,0.0)
        \psdots(0.6,0.6)
        \psdots(0,0.6)
        \psline(0.0,0.0)(0.6,0.0)
        \psline(0.6,0.0)(0.6,0.6)
        \psline(0.6,0.6)(0,0.6)
        \psline(0,0.6)(0.0,0.0)
	\uput{0.1}[0](0.6,0.3){\footnotesize $5$}
\end{pspicture} } 

\newcommand{\CTtwentyfour}{ 
\begin{pspicture}(-0.3,-0.3)(1.1,0.3)
        \psdots(-0.2,0.0)
        \psdots(0.2,0.0)
        \psline(-0.2,0.0)(0.2,0.0)
        \psdots(0.6,0.0)
        \psline(0.2,0.0)(0.6,0.0)
	\uput{0}[0](-0.1,0.2){\footnotesize $5$}
        \psdots(1.0,0.0)
        \psline(0.6,0.0)(1.0,0.0)
	\uput{0}[0](0.7,0.2){\footnotesize $5$}
\end{pspicture} } 

\newcommand{\CTtwentyfive}{ 
\begin{pspicture}(-0.1,-0.1)(0.8,0.8)
        \psdots(0.0,0.0)
        \psdots(0.6,0.0)
        \psdots(0.6,0.6)
        \psdots(0,0.6)
        \psline(0.0,0.0)(0.6,0.0)
        \psline(0.6,0.0)(0.6,0.6)
        \psline(0.6,0.6)(0,0.6)
        \psline(0,0.6)(0.0,0.0)
	\uput{0.1}[0](0.6,0.3){\footnotesize $4$}
\end{pspicture} } 

\newcommand{\CTtwentysix}{ 
\begin{pspicture}(-0.1,-0.1)(0.8,0.8)
        \psdots(0.0,0.0)
        \psdots(0.6,0.0)
        \psdots(0.6,0.6)
        \psdots(0,0.6)
        \psline(0.0,0.0)(0.6,0.0)
        \psline(0.6,0.0)(0.6,0.6)
        \psline(0.6,0.6)(0,0.6)
        \psline(0,0.6)(0.0,0.0)
	\uput{0.1}[180](0.0,0.3){\footnotesize $4$}
	\uput{0.1}[0](0.6,0.3){\footnotesize $5$}
\end{pspicture} } 

\newcommand{\CTtwentyseven}{ 
\begin{pspicture}(-0.3,-0.55)(0.6,0.55)
        \psdots(-0.25,0.433)
        \psdots(-0.25,-0.433)
        \psline(-0.25,0.433)(0.0,0.0)
        \psline(-0.25,-0.433)(0.0,0.0)
        \psdots(0.0,0.0)
        \psdots(0.5,0.0)
        \psline(0.5,0.0)(0.0,0.0)
	\uput{0}[0](0.25,0.2){\footnotesize $5$}
        \psline(-0.25,0.433)(-0.25,-0.433)
\end{pspicture} } 

\newcommand{\CTtwentyeight}{ 
\begin{pspicture}(-0.3,-0.3)(1.1,0.3)
        \psdots(-0.2,0.0)
        \psdots(0.2,0.0)
        \psline(-0.2,0.0)(0.2,0.0)
        \psdots(0.6,0.0)
        \psline(0.2,0.0)(0.6,0.0)
	\uput{0}[0](-0.1,0.2){\footnotesize $5$}
        \psdots(1.0,0.0)
        \psline(0.6,0.0)(1.0,0.0)
	\uput{0}[0](0.7,0.2){\footnotesize $6$}
\end{pspicture} } 

\newcommand{\CTtwentynine}{ 
\begin{pspicture}(-0.1,-0.1)(0.8,0.8)
        \psdots(0.0,0.0)
        \psdots(0.6,0.0)
        \psdots(0.6,0.6)
        \psdots(0,0.6)
        \psline(0.0,0.0)(0.6,0.0)
        \psline(0.6,0.0)(0.6,0.6)
        \psline(0.6,0.6)(0,0.6)
        \psline(0,0.6)(0.0,0.0)
	\uput{0.1}[180](0.0,0.3){\footnotesize $6$}
	\uput{0.1}[0](0.6,0.3){\footnotesize $5$}
\end{pspicture} } 

\newcommand{\CTthirty}{ 
\begin{pspicture}(-0.1,-0.1)(0.8,0.8)
        \psdots(0.0,0.0)
        \psdots(0.6,0.0)
        \psdots(0.6,0.6)
        \psdots(0,0.6)
        \psline(0.0,0.0)(0.6,0.0)
        \psline(0.6,0.0)(0.6,0.6)
        \psline(0.6,0.6)(0,0.6)
        \psline(0,0.6)(0.0,0.0)
	\uput{0.1}[180](0.0,0.3){\footnotesize $6$}
	\uput{0.1}[0](0.6,0.3){\footnotesize $4$}
\end{pspicture} } 

\newcommand{\CTthirtyone}{ 
\begin{pspicture}(-0.1,-0.1)(0.8,0.8)
        \psdots(0.0,0.0)
        \psdots(0.6,0.0)
        \psdots(0.6,0.6)
        \psdots(0,0.6)
        \psline(0.0,0.0)(0.6,0.0)
        \psline(0.6,0.0)(0.6,0.6)
        \psline(0.6,0.6)(0,0.6)
        \psline(0,0.6)(0.0,0.0)
	\uput{0.1}[180](0.0,0.3){\footnotesize $4$}
	\uput{0.1}[0](0.6,0.3){\footnotesize $4$}
	\uput{0.1}[270](0.3,0.0){\footnotesize $4$}
\end{pspicture} } 

\newcommand{\CTthirtytwo}{ 
\begin{pspicture}(-0.1,-0.1)(0.8,0.8)
        \psdots(0.0,0.0)
        \psdots(0.6,0.0)
        \psdots(0.6,0.6)
        \psdots(0,0.6)
        \psline(0.0,0.0)(0.6,0.0)
        \psline(0.6,0.0)(0.6,0.6)
        \psline(0.6,0.6)(0,0.6)
        \psline(0,0.6)(0.0,0.0)
	\uput{0.1}[0](0.6,0.3){\footnotesize $6$}
\end{pspicture} } 

\title[Accessing the cohomology of discrete groups above their vcd]{Accessing the cohomology of discrete groups \\ above their virtual cohomological dimension}
\author[Rahm]{Alexander D. Rahm}
\thanks{Funded by the Irish Research Council for Science, Engineering and Technology}
\address{National University of Ireland at Galway, Department of Mathematics}
\email{Alexander.Rahm@nuigalway.ie}
\urladdr{http://www.maths.nuigalway.ie/~rahm/}
\date{\today}
\begin{document}

\begin{abstract}
We introduce a method to explicitly determine the Farrell--Tate cohomology of discrete groups.
We apply this method to the Coxeter triangle and tetrahedral groups as well as to the Bianchi groups,
i.e. $\text{PSL}_2(\mathcal{O})$ for~$\mathcal{O}$ the ring of integers in an imaginary quadratic
number field, and to their finite index subgroups.
We show that the Farrell--Tate cohomology of the Bianchi groups is completely determined by the numbers
of conjugacy classes of finite subgroups.
In fact, our access to Farrell--Tate cohomology allows us to detach the information about it
 from geometric models for the Bianchi groups and to express it only with the group structure.
Formulae for the numbers of conjugacy classes of finite subgroups have been determined in
a thesis of Kr\"amer, in terms of elementary number-theoretic information on~$\mathcal{O}$.
 An evaluation of these formulae for a large number of Bianchi groups is provided numerically in the appendix.
Our new insights about their homological torsion allow us to give a conceptual description
 of the cohomology ring structure of the Bianchi groups.
\end{abstract}
\maketitle

\setcounter{secnumdepth}{3}
\setcounter{tocdepth}{1}

\section{Introduction}

Our objects of study are discrete groups~$\Gamma$ such that~$\Gamma$ admits a torsion-free subgroup of finite index. 
By a theorem of Serre, all the torsion-free subgroups of finite index in~$\Gamma$ have the same cohomological dimension;
this dimension is called the virtual cohomological dimension (abbreviated vcd) of~$\Gamma$.
Above the vcd, the (co)homology of a discrete group is determined by its system of finite subgroups.
We are going to discuss it in terms of Farrell--Tate cohomology (which we will by now just call Farrell cohomology).
The Farrell cohomology $\Farrell^q$  is  identical to group cohomology $\Homol^q$ in all degrees $q$ above the vcd, 
and extends in lower degrees to a cohomology theory of the system of finite subgroups.
Details are elaborated in~\cite{Brown}*{chapter X}.
So for instance considering the Coxeter groups, the virtual cohomological dimension of all of which vanishes, 
their Farrell cohomology is identical to all of their group cohomology. 
In Section~\ref{conjugacy reduction}, we will introduce a method of how to explicitly determine the Farrell cohomology : 
By reducing torsion sub-complexes.
This method has also been implemented on the computer \cite{HAP},
 which allows us to check the results that we obtain by our arguments.
We apply our method to the Coxeter triangle and tetrahedral groups in Section~\ref{Coxeter_groups},
 and to the Bianchi groups in Sections~\ref{The conjugacy classes of finite order elements}  through~\ref{cohomology ring}.

In detail, we require any discrete group $\Gamma$ under our study to be provided with a cell complex on which it acts cellularly. 
We call this a \emph{$\Gamma$--cell complex}.
Let $X$ be a $\Gamma$--cell complex; and let $\ell$ be a prime number.
 Denote by $X_{(\ell)}$ the set of all the cells $\sigma$ of $X$,
 such that there exists an element of order $\ell$ in the stabilizer of the cell $\sigma$.
In the case that the stabilizers are finite and fix their cells point-wise, the set $X_{(\ell)}$ is a $\Gamma$--sub-complex of $X$,
and we call it the \emph{$\ell$--torsion sub-complex}.

For the Coxeter tetrahedral groups, generated by the reflections on the sides of a tetrahedron in hyperbolic 3-space, we obtain the following. 
Denote by $\Dl$ the dihedral group of order $2\ell$.
\begin{corollary}[Corollary to Theorem~\ref{small rank Coxeter groups}.]
 Let $\Gamma$ be a Coxeter tetrahedral group, and  $\ell > 2$ be a prime number.
Then there is an isomorphism  
$\Homol_q(\Gamma; \thinspace \Z/\ell) \cong \left(\Homol_q(\Dl; \thinspace \Z/\ell)\right)^m$,
 with $m$ the number of connected components of the orbit space of the $\ell$--torsion sub-complex of the Davis complex of~$\Gamma$.
\end{corollary}
We specify the exponent $m$ in the tables in Figures~\ref{3-torsionCT1-14} through~\ref{5-torsionCT}.

Some individual procedures of our method have already been applied as ad hoc tricks by experts since~\cite{Soule},
usually without providing a printed explanation of the tricks.
An essential advantage of establishing a systematic method rather than using a set of ad hoc tricks, 
is that we can find ways to compute directly the quotient of the reduced torsion sub-complexes, 
working outside of the geometric model and skipping the often very laborious calculation of the orbit space of the
 $\Gamma$--cell complex.
This provides access to the cohomology of many discrete groups for which the latter orbit space calculation is far out of reach.
For the Bianchi groups, we give in Section~\ref{The conjugacy classes of finite order elements}
an instance of how to construct the quotient of the reduced torsion sub-complex outside of the geometric model.

\subsection*{Results for the Bianchi groups}
Denote by $\rationals(\sqrt{-m})$, with $m$ a square-free positive integer, an imaginary quadratic number field, and by $\ringO_{-m}$ its ring of integers.
The \emph{\mbox{Bianchi} groups} are the groups $\mathrm{PSL_2}(\ringO_{-m})$.
The \mbox{Bianchi} groups may be considered as a key to the study of a larger class of groups, the \emph{Kleinian} groups, which date back to work of Henri Poincar\'e~\cite{Poincare}.
In fact, each non-co-compact arithmetic Kleinian group is commensurable with some \mbox{Bianchi} group~\cite{MaclachlanReid}.
A wealth of information on the \mbox{Bianchi} groups can be found in the monographs \cite{Fine}, \cite{ElstrodtGrunewaldMennicke}, \cite{MaclachlanReid}.
Kr\"amer \cite{Kraemer} has determined number-theoretic formulae for the numbers of conjugacy classes of finite subgroups in the Bianchi groups, using numbers of ideal classes in
orders of cyclotomic extensions of $\rationals(\sqrt{-m})$. 

In Section~\ref{The Kraemer numbers and group homology},
 we express the homological torsion of the Bianchi groups as a function of these numbers of conjugacy classes.
To achieve this, we build on the geometric techniques of \cite{Rahm_homological_torsion},
which depend on the explicit knowledge of the quotient space of geometric models for the Bianchi groups ---
like any technique effectively accessing the (co)homology of the Bianchi groups, either directly \cite{SchwermerVogtmann},
 \cite{Vogtmann} or via a group presentation \cite{BerkoveMod2}.
For the Bianchi groups, we can in Sections~\ref{The conjugacy classes of finite order elements}
 and~\ref{The Kraemer numbers and group homology} detach invariants of the group actions from the geometric models,
 in order to express them only by the group structure itself, in terms of conjugacy classes of finite subgroups,
 normalizers of the latter, and their interactions.
This information is already contained in our reduced torsion sub-complexes.

Not only does this provide us  with exact formulae for the homological torsion of the Bianchi groups, 
the power of which we can see in the numerical evaluations of Appendices~\ref{Numerical evaluation of Kraemer's formulae in 3-torsion}
 and~\ref{Numerical evaluation of Kraemer's formulae in 2-torsion},
also it allows us to understand the r\^ole of the centralizers of the finite subgroups ---
and this is how in~\cite{orbifold_cohomology}, some more fruits of the present results are harvested
(in terms of the Chen/Ruan orbifold cohomology of the orbifolds given by the action of the Bianchi groups on complexified hyperbolic space).

Except for the Gauss{}ian and Eisenstein integers, which can easily be treated separately \cite{SchwermerVogtmann}, \cite{Rahm_homological_torsion}, all the rings of integers of imaginary quadratic number fields admit as only units $\{\pm 1\}$. In the latter case, we call $\PSL_2(\ringO)$ a \textit{Bianchi group with units} $\{\pm 1\}$.
For the possible types of finite subgroups in the Bianchi groups, see Lemma~\ref{finiteSubgroups} : 
There are five non-trivial possibilities.
In Theorem~\ref{Grunewald-Poincare series formulae}, the proof of which we give in
 Section~\ref{The Kraemer numbers and group homology}, we give a formula expressing precisely how the Farrell cohomology
 of the Bianchi groups with units $\{\pm 1\}$
 depends on the numbers of conjugacy classes of non-trivial finite subgroups of the occurring five types.
The main step in order to prove this, is to read off the Farrell cohomology from the quotient of the reduced torsion sub-complexes.

Kr\"amer's formulae express the numbers of conjugacy classes of the five types of non-trivial finite subgroups in the Bianchi groups, where the symbols in the first row are Kr\"amer's notations for the number of their conjugacy classes:
$$\begin{array}{|c|c|c|c|c|}
\hline \lambda_{4} &  \lambda_6  &  \mu_2      & \mu_3   & \mu_T   \\
\hline \Z/2         & \Z/3 &   \Kleinfourgroup & \Sthree & \Afour   \\
\hline
  \end{array}$$
We are going to use these symbols also for the numbers of conjugacy classes in $\Gamma$,
 where $\Gamma$ is a finite index subgroup in a Bianchi group.
Recall that for $\ell = 2$ and $\ell = 3$, 
we can express the the dimensions of the homology of $\Gamma$ with coefficients in the field ${\mathbb{F}_\ell}$ with $\ell$ elements
 in degrees above the virtual cohomological dimension of the Bianchi groups -- which is $2$ -- by the Poincar\'e series
$$P^\ell_\Gamma(t) := \sum\limits_{q \thinspace > \thinspace 2}^{\infty} \dim_{\mathbb{F}_\ell} \Homol_q \left(\Gamma;\thinspace {\mathbb{F}_\ell} \right)\thinspace t^q,$$
which has been suggested by Grunewald.
Further let $P_{\circlegraph} (t) := \frac{-2t^3}{t-1}$ , which equals the series $P^2_\Gamma(t)$  of the groups $\Gamma$ the quotient of the reduced $2$--torsion sub-complex of which is a circle.
Denote by \begin{itemize}
\item $P_{\Kleinfourgroup}^*(t) := \frac{-t^3(3t -5)}{2(t-1)^2}$, the Poincar\'e series over $\dim_{\mathbb{F}_2} \Homol_q \left(\Kleinfourgroup;\thinspace {\mathbb{F}_2} \right) -\frac{3}{2}\dim_{\mathbb{F}_2} \Homol_q \left(\Z/2;\thinspace {\mathbb{F}_2} \right)$ 
\item and by $P_{\Afour}^*(t) := \frac{-t^3(t^3 - 2t^2 + 2t - 3)}{2(t-1)^2 (t^2 + t + 1 ) }$, the Poincar\'e series over 
$$\dim_{\mathbb{F}_2} \Homol_q \left(\Afour;\thinspace {\mathbb{F}_2} \right) -\frac{1}{2}\dim_{\mathbb{F}_2} \Homol_q \left(\Z/2;\thinspace {\mathbb{F}_2} \right).$$
\end{itemize}

In 3-torsion, let
$P_{\edgegraph} (t) := \frac{-t^3(t^2 - t + 2)}{(t-1)(t^2+1)}$, which equals the series $P^3_\Gamma(t)$ for the Bianchi groups the quotient of the reduced $3$--torsion sub-complex of which is a single edge without identifications.

\begin{theorem} \label{Grunewald-Poincare series formulae}
For any finite index subgroup $\Gamma$ in a Bianchi group with units $\{\pm 1\}$, the group homology in degrees above its virtual cohomological dimension is given by the Poincar\'e series
$$P^2_\Gamma(t) = \left(\lambda_4 -\frac{3\mu_2 -2\mu_T}{2}\right)P_{\circlegraph} (t) +(\mu_2 -\mu_T)P_{\Kleinfourgroup}^*(t) +\mu_T P_{\Afour}^*(t)$$
and
$$P^3_\Gamma(t) =  \left(\lambda_6 -\frac{\mu_3}{2}\right)P_{\circlegraph} (t) + \frac{\mu_3}{2}P_{\edgegraph}(t).$$
\end{theorem}

Our method is further applied in~\cite{BerkoveRahm} to obtain also the Farrell cohomology of $\SLO$.

\subsection*{Organization of the paper}
In Section~\ref{conjugacy reduction}, we introduce our method to explicitly determine Farrell cohomology: 
By reducing the torsion sub-complexes.
We apply our method to the Coxeter triangle and tetrahedral groups in Section~\ref{Coxeter_groups}.
In Section~\ref{The conjugacy classes of finite order elements},
 we show how to read off the Farrell cohomology of the Bianchi groups from the reduced torsion sub-complexes.
 We achieve this by showing that for the Bianchi groups,
 the quotients of the reduced torsion sub-complexes are homeomorphic to conjugacy classes graphs
 that we can define without reference to any geometric model.
This enables us in Section~\ref{The Kraemer numbers and group homology} to prove the formulae for the homological torsion
 of the Bianchi groups in terms of numbers of conjugacy classes of finite subgroups.
We use this to establish the structure of the classical cohomology rings of the Bianchi groups in Section~\ref{cohomology ring}.
Kr\"amer has given number-theoretic formulae for these numbers of conjugacy classes,
 and we evaluate them numerically in Appendices~\ref{Numerical evaluation of Kraemer's formulae in 3-torsion}
 and~\ref{Numerical evaluation of Kraemer's formulae in 2-torsion}.
Finally, we present some numerical asymptotics on the numbers of conjugacy classes
 in Appendix~\ref{Asymptotic behaviour of the number of conjugacy classes}.

\subsection*{Acknowledgements} The author is indebted to the late great mathematician Fritz Grunewald,
 for telling him about the existence and providing him a copy of Kr\"amer's Diplom thesis.
 Warmest thanks go to Rub{\'e}n S{\'a}nchez-Garc{\'{\i}}a for providing his implementation of the Davis complex,
 to Mike Davis and G\"otz Pfeiffer for discussions on the Coxeter groups,
 to Oliver Braunling for a correspondence on the occurrence of given norms on rings of integers,
 to Nicolas Bergeron for discussions on asymptotics,
 to Philippe Elbaz-Vincent and Matthias Wendt for a very careful lecture of the manuscript and helpful suggestions,
 and to Graham Ellis and \mbox{Stephen S. Gelbart} for support and encouragement.

\section{Reduction of torsion sub-complexes} \label{conjugacy reduction}
Let $X$ be a finite-dimensional cell complex with a cellular action of a discrete group~$\Gamma$,
 such that each cell stabilizer fixes its cell point-wise.
Let $\ell$ be a prime such that every non-trivial finite $\ell$--subgroup of~$\Gamma$ admits a contractible fixed point set.
 \textit{We keep these requirements on the $\Gamma$--action as a general assumption throughout this article.}
Then, the $\Gamma$--equivariant Farrell cohomology of~$X$, for any trivial $\Gamma$--module $M$ of coefficients, gives us the 
$\ell$--primary part $\Farrell^*(\Gamma; \thinspace M)_{(\ell)}$ of the Farrell cohomology of~$\Gamma$, as follows.
\begin{proposition}[Brown \cite{Brown}] \label{Brown's proposition}
Under our general assumption, the canonical map
$$ \Farrell^*(\Gamma; \thinspace M)_{(\ell)} \to \Farrell^*_\Gamma(X; \thinspace M)_{(\ell)} $$
is an isomorphism.
\end{proposition}

The classical choice \cite{Brown} is to take for $X$
 the geometric realization of the partially ordered set of non-trivial finite subgroups 
(respectively, non-trivial elementary Abelian $\ell$--subgroups) of~$\Gamma$,
 the latter acting by conjugation. The stabilizers are then the normalizers, which in many discrete groups are infinite.
And it can impose great computational challenges to determine a group presentation for them.
When we want to compute the module $\Farrell^*_\Gamma(X; \thinspace M)_{(\ell)}$
subject to Proposition~\ref{Brown's proposition},
at least we must get to know the ($\ell$--primary part of the) Farrell cohomology of these normalizers.
The Bianchi groups are an instance that different isomorphism types can occur for this cohomology
 at different conjugacy classes of elementary Abelian $\ell$--subgroups, both for $\ell=2$ and $\ell=3$.
As the only non-trivial elementary Abelian $3$--subgroups in the Bianchi groups are of rank $1$,
the orbit space $_\Gamma \backslash X$ consists only of one point for each conjugacy class of type $\Z/3$
and a corollary~\cite{Brown} from Proposition~\ref{Brown's proposition} decomposes the
 $3$--primary part of the Farrell cohomology of the Bianchi groups into the direct product over their normalizers.
However, due to the different possible homological types of the normalizers (in fact, two of them occur),
 the final result remains unclear and subject to tedious case-by-case computations of the normalizers.

In contrast, in the cell complex we are going to develop,
 the connected components of the orbit space are for the $3$--torsion in the Bianchi groups not simple points,
 but have either the shape $\edgegraph$ or $\circlegraph$.
This dichotomy already contains the information about the occurring normalizer. 

\begin{df}
Let $\ell$ be a prime number. The \emph{$\ell$--torsion sub-complex} of the $\Gamma$--cell complex $X$ consists of all the cells of $X$
 the stabilizers in~$\Gamma$  of which contain elements of order $\ell$.
\end{df}

We are from now on going to require the cell complex $X$ to admit only finite stabilizers in~$\Gamma$, and we require the action of $\Gamma$ on the coefficient module $M$ to be trivial.
Then obviously only cells from the \emph{$\ell$--torsion sub-complex} contribute to $\Farrell^*_\Gamma(X; \thinspace M)_{(\ell)}$.
We are going to reduce the \emph{$\ell$--torsion sub-complex} to one which still carries the 
$\Gamma$--equivariant Farrell cohomology of~$X$,
but can have considerably less orbits of cells, can be easier to handle in practice,
and, for certain classes of groups, leads us to an explicit structural description of the Farrell cohomology of~$\Gamma$.
The pivotal property of this reduced $\ell$--torsion sub-complex will be given in Theorem~\ref{pivotal}.
 
\begin{ConditionA} \label{cell condition}
In the $\ell$--torsion sub-complex, let $\sigma$ be a cell of dimension $n-1$, lying in the boundary of precisely two $n$--cells $\tau_1$ and~$\tau_2$,
 the latter cells representing two different orbits.
Assume further that no higher-dimensional cells of the $\ell$--torsion sub-complex touch $\sigma$;
and that the $n$--cell stabilizers admit an isomorphism
$\Gamma_{\tau_1} \cong \Gamma_{\tau_2}$. 
\end{ConditionA}

Where this condition is fulfilled in the $\ell$--torsion sub-complex,
 we merge the cells $\tau_1$ and $\tau_2$ along~$\sigma$ and do so for their entire orbits,
 if and only if they meet the following additional condition. 
We never merge two cells the interior of which contains two points on the same orbit.
Let $\ell$ be a prime number, and denote by \emph{mod $\ell$ homology}
 group homology with $\Z/\ell$--coefficients under the trivial action.

\begin{isomorphismCondition} 
The inclusion $ \Gamma_{\tau_1} \subset \Gamma_\sigma$ induces an isomorphism on mod $\ell$ homology.
\end{isomorphismCondition}

\begin{lemma} \label{A}
 Let $\widetilde{X_{(\ell)}}$ be the $\Gamma$--complex obtained by orbit-wise merging two $n$--cells of the
 $\ell$--torsion sub-complex $X_{(\ell)}$
which satisfy Conditions~$\cellCondition$ and~$\IsomorphismConditionSymbol$.
Then, $$\Farrell^*_\Gamma(\widetilde{X_{(\ell)}}; \thinspace M)_{(\ell)} \cong \Farrell^*_\Gamma(X_{(\ell)}; \thinspace M)_{(\ell)}.$$
\end{lemma}

\begin{proof}[Proof of Lemma \emph{\ref{A}}]
 Consider the equivariant spectral sequence in Farrell cohomology \cite{Brown}.
On the $\ell$--torsion sub-complex, it includes a map
$$\xymatrix{ \Farrell^*(\Gamma_\sigma; \thinspace M)_{(\ell)}  
\ar[rrr]^{d_1^{(n-1),*}|_{\Farrell^*(\Gamma_\sigma; \thinspace M)_{(\ell)}} \qquad \qquad} &&&
 \Farrell^*(\Gamma_{\tau_1}; \thinspace M)_{(\ell)} \oplus \Farrell^*(\Gamma_{\tau_2}; \thinspace M)_{(\ell)} },$$
which is the diagonal map with blocks the isomorphisms
$\xymatrix{ \Farrell^*(\Gamma_\sigma; \thinspace M)_{(\ell)} \ar[r]^{\cong} & 
\Farrell^*(\Gamma_{\tau_i}; \thinspace M)_{(\ell)} },$
induced by the inclusions $\Gamma_{\tau_i} \hookrightarrow \Gamma_\sigma$.
The latter inclusions are required to induce isomorphisms in Condition~$\IsomorphismConditionSymbol$. 
If for the orbit of $\tau_1$ or $\tau_2$ we have chosen a representative which is not adjacent to $\sigma$,
then this isomorphism is composed with the isomorphism induced by conjugation with the element of~$\Gamma$
carrying the cell to one adjacent to $\sigma$. 
Hence, the map $d_1^{(n-1),*}|_{\Farrell^*(\Gamma_\sigma; \thinspace M)_{(\ell)}}$ has vanishing kernel,
 and dividing its image out of $ \Farrell^*(\Gamma_{\tau_1}; \thinspace M)_{(\ell)} \oplus \Farrell^*(\Gamma_{\tau_2}; \thinspace M)_{(\ell)} $
gives us the $\ell$--primary part $ \Farrell^*(\Gamma_{\tau_1 \cup \tau_2}; \thinspace M)_{(\ell)}$ 
of the Farrell cohomology of the union $\tau_1 \cup \tau_2$ of the two $n$--cells,
once that we make use of the isomorphism $\Gamma_{\tau_1} \cong \Gamma_{\tau_2}$ of Condition~$\cellCondition$.
As by Condition~$\cellCondition$ no higher-dimensional cells are touching $\sigma$, 
there are no higher degree differentials interfering.
\end{proof}

By a ``terminal vertex'',
 we will denote a vertex with no adjacent higher-dimensional cells and precisely one adjacent edge in the quotient space,
 and by ``cutting off'' the latter edge,
 we will mean that we remove the edge together with the terminal vertex from our cell complex.

\begin{df}
 The \emph{reduced $\ell$--torsion sub-complex} associated to a $\Gamma$--cell complex~$X$ which fulfills our general assumption,
 is the cell complex obtained by recursively merging orbit-wise all the pairs of cells satisfying 
 Conditions~$\cellCondition$ and~$\IsomorphismConditionSymbol$;
 and cutting off edges that admit a terminal vertex together with which they satisfy Condition~$\IsomorphismConditionSymbol$.
\end{df}

\begin{theorem} \label{pivotal}
 There is an isomorphism between the $\ell$--primary parts of the Farrell cohomology of~$\Gamma$ and the
 $\Gamma$--equivariant Farrell cohomology of the reduced $\ell$--torsion sub-complex.
\end{theorem}
\begin{proof}
 We apply Proposition~\ref{Brown's proposition} to the cell complex $X$,
 and then we apply Lemma~\ref{A} each time that we orbit-wise merge a pair of cells of the $\ell$--torsion sub-complex,
 or that we cut off an edge.
\end{proof}

In order to have a practical criterion for checking Condition~$\IsomorphismConditionSymbol$, 
we make use of the following stronger condition.

Here, we write ${\rm N}_{\Gamma_\sigma}$ for taking the normalizer in ${\Gamma_\sigma}$ and 
${\rm Sylow}_\ell$ for picking an arbitrary Sylow $\ell$--subgroup.
 This is well defined because all Sylow $\ell$--subgroups are conjugate.
We use Zassenhaus's notion for a finite group to be $\ell$--\emph{normal},
 if the center of one of its Sylow $\ell$--subgroups is the center of every Sylow $\ell$--subgroup in which it is contained.

\begin{ConditionBprime} 
The group $\Gamma_\sigma$ admits a (possibly trivial) normal subgroup $T_\sigma$ with trivial mod~$\ell$ homology
and with quotient group $G_\sigma$; and the group $\Gamma_{\tau_1}$ admits a (possibly trivial) normal subgroup 
$T_\tau$ with trivial mod~$\ell$ homology and with quotient group $G_\tau$ making the sequences 
\begin{center}
 $ 1 \to T_\sigma \to \Gamma_\sigma \to G_\sigma \to 1$ and $ 1 \to T_\tau \to \Gamma_{\tau_1} \to G_\tau \to 1$
\end{center}
exact and satisfying one of the following.
\begin{enumerate}
 \item  Either $G_\tau \cong G_\sigma$, or
 \item $G_\sigma$ is $\ell$--normal and $G_\tau \cong {\rm N}_{G_\sigma}({\rm center}({\rm Sylow}_\ell(G_\sigma)))$, or
 \item both $G_\sigma$  and $G_\tau$ are $\ell$--normal and there is a (possibly trivial) group $T$
 with trivial mod~$\ell$ homology making the sequence
$$1 \to T \to {\rm N}_{G_\sigma}({\rm center}({\rm Sylow}_\ell(G_\sigma))) \to {\rm N}_{G_\tau}({\rm center}({\rm Sylow}_\ell(G_\tau))) \to 1$$
exact.
\end{enumerate}
\end{ConditionBprime}

\begin{lemma} \label{Implying the isomorphism condition}
Condition B' implies Condition B.
\end{lemma}

For the proof of ( B'(2) $\Rightarrow$ B), we use Swan's extension \cite{Swan1960}*{final corollary}
 to Farrell cohomology of the Second Theorem of Gr\"un \cite{Gruen}*{Satz 5}.
\begin{theorem}[Swan] \label{Gruen-Swan}
 Let $G$ be a $\ell$--normal finite group, and let $N$ be the normalizer of  the center of a Sylow $\ell$--subgroup of $G$.
Let $M$ be any trivial $G$--module. Then the inclusion and transfer maps both are isomorphisms between the $\ell$--primary components of
$ \Farrell^*(G; \thinspace M)$ and  $\Farrell^*(N; \thinspace M)$.
\end{theorem}

For the proof of ( B'(3) $\Rightarrow$ B),
 we make use of the following direct consequence of the Lyndon--Hochschild--Serre spectral sequence. 
\begin{lemma} \label{extension}
 Let $T$ be a group with trivial mod $\ell$ homology, and consider any group extension
$$ 1 \to T \to E \to Q \to 1.$$
Then the map $E \to Q$ induces an isomorphism on mod $\ell$ homology.
\end{lemma}
This statement may look like a triviality,
 but it becomes wrong as soon as we exchange the r\^oles of $T$ and $Q$ in the group extension.
In degrees $1$ and $2$, our claim follows from \cite{Brown}*{VII.(6.4)}. 
In arbitrary degree, it is more or less known and we can proceed through the following easy steps.
\begin{proof}
 Consider the Lyndon--Hochschild--Serre spectral sequence associated to the group extension, namely
\begin{center}
 $ E^2_{p,q} = \Homol_p( Q; \thinspace \Homol_q(T; \thinspace \Z/\ell))$ converges to
 $ \Homol_{p+q}(E; \thinspace \Z/\ell).$
\end{center}
By our assumption, $\Homol_q(T; \thinspace \Z/\ell)$ is trivial,
so this spectral sequence concentrates in the row $q=0$, degenerates on the second page and yields isomorphisms
\begin{equation}\label{star}
 \Homol_p( Q; \thinspace \Homol_0(T; \thinspace \Z/\ell)) \cong \Homol_{p}(E; \thinspace \Z/\ell).
\end{equation}
As for the modules of co-invariants, we have $\left( (\Z/\ell)_T \right)_Q \cong (\Z/\ell)_E$ \medspace \quad \cite{McCleary},
the trivial actions of $E$ and $T$ induce that also the action of $Q$ on the coefficients in
 $\Homol_0(T; \thinspace \Z/\ell)$ is trivial.
Thus, Isomorphism~(\ref{star}) becomes
$\Homol_p( Q; \thinspace \Z/\ell) \cong \Homol_{p}(E; \thinspace \Z/\ell).$
\end{proof}
The above lemma directly implies that any extension of two groups both having trivial mod~$\ell$ homology,
 again has trivial mod $\ell$ homology.

\begin{proof}[Proof of Lemma~\emph{\ref{Implying the isomorphism condition}}]
 We combine Theorem \ref{Gruen-Swan} and Lemma \ref{extension} in the obvious way.
\end{proof}

\begin{remark}
 The computer implementation \cite{HAP} checks Conditions~$\technicalCondition (1)$ and $\technicalCondition (2)$ for each pair of cell stabilizers,
using a presentation of the latter in terms of matrices, permutation cycles or generators and relators.
In the below examples however,
 we do avoid this case-by-case computation by a general determination of the isomorphism types of pairs of cell stabilizers
 for which group inclusion induces an isomorphism on mod $\ell$ homology.
The latter method is to be considered as the procedure of preference, 
because it allows us to deduce statements that hold for the whole class of concerned groups.
\end{remark}

\section{Farrell cohomology of the Coxeter tetrahedral groups} \label{Coxeter_groups}
\begin{wrapfigure}[16]{r}[15pt]{50mm}
\begin{center}
\includegraphics[scale = 0.5]{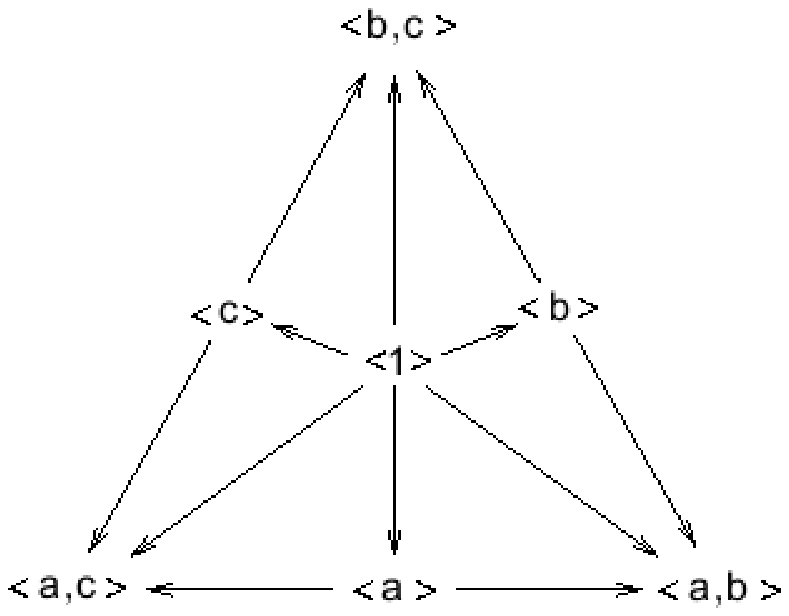}
\caption{Quotient of the Davis complex for a triangle group (diagram reprinted with the kind permission of Sanchez-Garcia \cite{Sanchez-Garcia_Coxeter}).}
\label{fig:trianglesd}
\end{center}
\end{wrapfigure}
Recall that a Coxeter group is a group admitting a presentation
$$\langle g_1, g_2, ..., g_n \medspace | \medspace (g_i g_j)^{m_{i,j}} = 1 \rangle,$$
where $m_{i,i} = 1$; for $i \neq j$ we have $m_{i,j} \geq 2$;
 and $m_{i,j} = \infty$ is permitted, meaning that $(g_i g_j)$ is not of finite order.
As the Coxeter groups admit a contractible classifying space for proper actions \cite{Davis},
 their Farrell cohomology yields all of their group cohomology.
So in this section, we make use of this fact to determine the latter.
For facts about Coxeter groups, and especially for the Davis complex, we refer to \cite{Davis}.
Recall that the simplest example of a Coxeter group, the dihedral group $\Dn$, is an extension
$$ 1 \to \Z/\ell \to \Dn \to \Z/2 \to 1,$$
so we can make use of the original application~\cite{Wall} of Wall's lemma to obtain its mod $\ell$ homology for prime numbers
 $\ell >2$, 
$$ \Homol_q(\Dn; \thinspace \Z/\ell) \cong 
\begin{cases}
				\Z/\ell, & q = 0, \\
                               \Z/{\rm gcd}(n,\ell), & q \equiv 3 \medspace {\rm or} \medspace 4 \mod 4, \\
				0, & {\rm otherwise}.
\end{cases}
$$
\begin{theorem} \label{small rank Coxeter groups}
 Let $\ell > 2$ be a prime number.
 Let $\Gamma$ be a Coxeter group admitting a Coxeter system with at most four generators,
 and relator orders not divisible by~$\ell^2$.
Let $Z_{(\ell)}$ be the $\ell$--torsion sub-complex of the Davis complex of~$\Gamma$.
If $Z_{(\ell)}$ is at most one-dimensional and its orbit space contains no loop nor bifurcation, then the$\mod \ell$ homology of~$\Gamma$ is isomorphic to 
$\left(\Homol_q(\Dl; \thinspace \Z/\ell)\right)^m$, with $m$ the number of connected components of the orbit space of~$Z_{(\ell)}$.
\end{theorem}
The conditions of this theorem are for instance fulfilled by the Coxeter tetrahedral groups;
 we specify the exponent $m$ for them in the tables in Figures~\ref{3-torsionCT1-14}
 through~\ref{5-torsionCT}.
In order to prove Theorem~\ref{small rank Coxeter groups}, we lean on the following technical lemma.
When a group $G$ contains a Coxeter group $H$ properly (i.e. $H \neq G$) as a subgroup, then we call $H$ a Coxeter subgroup of $G$. 
\begin{lemma} \label{finite Coxeter subgroups}
 Let $\ell > 2$ be a prime number; and let $\Gamma_\sigma$ be a finite Coxeter group  with
 $n \leq 4$  generators.
 If $\Gamma_\sigma$ is not a direct product of two dihedral groups and not associated to the Coxeter diagram \emph{\textbf{F}}$_4$
 or \emph{\textbf{H}}$_4$, 
then Condition~$\technicalCondition$ is fulfilled for the triple consisting of $\ell$,
 the group $\Gamma_\sigma$ and any of its Coxeter subgroups $\Gamma_{\tau_1}$
 with $(n-1)$ generators that contains $\ell$--torsion elements. 
\end{lemma}
\begin{proof}
The dihedral groups admit only Coxeter subgroups with two elements, so without $\ell$--torsion. 
 There are only finitely many other isomorphism types of irreducible finite Coxeter groups with at most four generators,
 specified by the Coxeter diagrams 
$$
\begin{array}{|c|c|c|c|c|c|c|}
\hline &&&&&& \\
{\mathrm{\mathbf A}}_1  & {\mathrm {\mathbf A}}_3           &  {\mathrm {\mathbf A}}_4  & {\mathrm {\mathbf B}}_3 & {\mathrm {\mathbf B}}_4 & {\mathrm {\mathbf D}}_4  & {\mathrm {\mathbf H}}_3  \\
\hline &&&&&& \\
\aTop{\bullet}{}       & \aTop{\CoxeterDiagramOfSfour}{} & \CoxeterDiagramOfSfive  & \aTop{\Bthree}{}      & \Bfour                &  \CoxeterDiagramDfour  & \aTop{\Hthree}{} \\
\hline
  \end{array}
$$
 on which we can check the condition case by case.
\begin{itemize}
 \item[\textbf{A}$_1$.] The symmetric group $\Stwo$ admits no Coxeter subgroups. 
 \item[\textbf{A}$_3$.] The symmetric group $\Sfour$ is $3$--normal; and its Sylow-$3$--subgroups are of type $\Z/3$,
so they are identical to their center.
Their normalizers in~$\Sfour$ match the Coxeter subgroups of type
 $\Dthree$ that one obtains by omitting one of the generators of 
$\Sfour$ at an end of its Coxeter diagram.
The other possible Coxeter subgroup type is $(\Z/2)^2$, obtained by omitting the middle generator in this diagram,
 and contains no $3$--torsion.
 \item[\textbf{A}$_4$.] The Coxeter subgroups with three generators in the symmetric group $\Sfive$ are 
$\Dthree \times \Z/2$ and~$\Sfour$, so we only need to consider $3$--torsion.
The group $\Sfive$ is $3$--normal; the normalizer of the center of any of its Sylow-$3$--subgroups is of type $\Dthree \times \Z/2$.
So for the Coxeter subgroup~$\Sfour$, we use the normalizer $\Dthree$ of its Sylow-$3$--subgroup $\Z/3$;
 and see that Condition~$\technicalCondition (3)$ is fulfilled.
 \item[\textbf{B}$_3$.] We apply Lemma~\ref{extension} to the Coxeter group $(\Z/2)^3 \rtimes \Sthree$,
 and retain only $\Sthree$, which is isomorphic to the only Coxeter subgroup admitting $3$--torsion. 
 \item[\textbf{B}$_4$.] The Coxeter subgroups with three generators are of type $\Sfour$, $\Z/2 \times \Dthree$,
$\Dfour \times \Z/2$ or $(\Z/2)^3 \rtimes \Sthree$, thus for the three of them containing $3$--torsion,
 we use the above methods to relate them to $\Dthree$. The Coxeter group $(\Z/2)^4 \rtimes \Sfour$ is $3$--normal; 
its Sylow-$3$--subgroup is of type $\Z/3$ and admits a normalizer $N$ fitting into the exact sequence
$$1 \to (\Z/2)^2 \to N \to \Dthree \to 1.$$
 \item[\textbf{D}$_4$.] From the Coxeter diagram, we see that the Coxeter subgroups with three generators are $(\Z/2)^3$ and $\Sfour$.
So we only need to compare with the $3$--torsion of $\Sfour$.
For this purpose, we apply Lemma~\ref{extension} to the Coxeter group $(\Z/2)^3 \rtimes \Sfour$. 
 \item[\textbf{H}$_3$.] The symmetry group $\Icos$
of the icosahedron splits as a direct product $\Z/2 \times \Afive$, so by Lemma~\ref{extension},
we can for all primes $\ell>2$ make use of the alternating group $\Afive$ as the quotient group in Condition~$\technicalCondition$.
The primes other than $2$, at which the homology of $\Afive$ admits torsion, are $3$ and $5$.
So now let $\ell$ be $3$ or $5$.
Then the group $\Afive$ is $\ell$--normal; and its Sylow-$\ell$--subgroups are of type $\Z/\ell$,
so they are identical to their center. Their normalizers in $\Afour$ are of type $\Dl$.
From the Coxeter diagram, we see that this is the only Coxeter subgroup type with two generators that contains $\ell$--torsion.
\end{itemize}
The case where we have a direct product of the one-generator Coxeter group $\Z/2$ with one of the above groups,
 is already absorbed by Condition~$\technicalCondition$.
\end{proof}
\begin{proof}[Proof of Theorem~\emph{\ref{small rank Coxeter groups}}.]
 The Davis complex is a finite-dimensional cell complex with a cellular action of the Coxeter group~$\Gamma$
 with respect to which it is constructed,
 such that each cell stabilizer fixes its cell point-wise.
 Also, it admits the property that the fixed point sets of the finite subgroups of~$\Gamma$ are acyclic \cite{Davis}.
 Thus by Proposition~\ref{Brown's proposition}, the $\Gamma$--equivariant Farrell cohomology of the Davis complex gives us the 
 $\ell$--primary part of the Farrell cohomology of~$\Gamma$.
 As the $3$--torsion sub-complex for the group generated by the Coxeter diagram \textbf{F}$_4$ (the symmetry group of the $24$--cell)
 and the $3$-- and $5$--torsion sub-complexes for the group generated by the Coxeter diagram \textbf{H}$_4$ (the symmetry group of the $600$--cell)
 as well as the $\ell$--torsion sub-complex of a direct product of two dihedral groups with $\ell$--torsion
 all contain $2$--cells, we are either in the case where the
 $\ell$--torsion sub-complex is trivial or in the case in which we suppose to be from now on,
 namely where $\Gamma$ is not one of the groups just mentioned. 
 Then all the finite Coxeter subgroups of~$\Gamma$ fulfill the hypothesis of Lemma~\ref{finite Coxeter subgroups},
 and hence all pairs of a vertex stabilizer and the stabilizer of an adjacent edge satisfy Condition~$\technicalCondition$.
 By the assumptions on~$Z_{(\ell)}$, also Condition~$\cellCondition$ is fulfilled for any pair of adjacent edges in $Z_{(\ell)}$.
 Hence, every connected component of the reduced $\ell$--torsion sub-complex is a single vertex.
 From recursive use of Lemma~\ref{finite Coxeter subgroups} and the assumption that the relator orders are not divisible by~$\ell^2$,
 we see that the stabilizer of the latter vertex has the mod~$\ell$ homology of $\Dl$.
 Theorem~\ref{pivotal} now yields our claim.
\end{proof}

Let us determine the exponent $m$ of Theorem~\ref{small rank Coxeter groups} for some classes of examples.
\\
 The \emph{Coxeter triangle groups} are given by the presentation
$$
    \langle\, a, b, c \;|\; a^2 = b^2 = c^2 = (ab)^p = (bc)^q =
    (c a)^r = 1 \,\rangle\, ,
$$
where $2 \leq p,q,r \in \N$ and $\frac{1}{p} + \frac{1}{q} +
\frac{1}{r} \le 1$. 

\vbox{
\begin{proposition}
For any prime $\ell>2$, the$\mod \ell$ homology of a Coxeter triangle group is given as the direct sum over 
the$\mod \ell$ homology of the dihedral groups ${\mathcal D}_p$, ${\mathcal D}_q$ and ${\mathcal D}_r$.
\end{proposition}
\begin{proof}
The quotient space of the Davis complex of a Coxeter triangle group can be realized as the barycentric
subdivision of an Euclidean or hyperbolic triangle with interior
angles $\frac{\pi}{p}$, $\frac{\pi}{q}$ and $\frac{\pi}{r}$, and
$a$, $b$ and $c$ acting as reflections through the corresponding
sides. 

We obtain this triangle by realizing the partially ordered set (where arrows stand for inclusions)
 of Figure~\ref{fig:trianglesd}.
 The whole Davis complex of the Coxeter triangle groups is then given as a tessellation of the Euclidean
or hyperbolic plane by these triangles.
The quotient space of the $\ell$--torsion sub-complex then consists of one vertex for each of the dihedral groups 
${\mathcal D}_p$, ${\mathcal D}_q$ and~${\mathcal D}_r$ which contain an element of order $\ell$.
Theorem~\ref{pivotal} now yields the result.
\end{proof}
}

\subsection{Results for the Coxeter tetrahedral groups}

\begin{figure}
 $$\begin{array}{|c|l|c|c|}
\hline  & &&\\
{\rm Name}\medspace {\rm and} \medspace {\rm Coxeter \medspace graph}
 & \begin{array}{c} 5{\rm -torsion}\\ {\rm subcomplex \medspace quotient}\end{array}
 & \begin{array}{c}{\rm reduced  \medspace} 5{\rm -torsion}\\ {\rm subcomplex \medspace quotient}\end{array}
 & \Homol_q(CT(m);\thinspace {\mathbb{F}_5}) \\
\hline  & &&\\ 
\aTop{CT(19),}{\CTnineteen} \aTop{CT(28)}{\CTtwentyeight} &  \ftsNineteen 
& \aTop{\bullet \Dfive}{} & \aTop{\Homol_q(\Dfive;\thinspace {\mathbb{F}_5}) }{}
\\ 
\hline  & &&\\
\aTop{\aTop{CT(20),}{\CTtwenty} \aTop{CT(22),}{ \CTtwentytwo} \aTop{CT(23),}{\CTtwentythree}}{ 
  \aTop{CT(26),}{ \CTtwentysix} \aTop{CT(27),}{\CTtwentyseven} \aTop{CT(29)}{\CTtwentynine} }
& \ftsTwenty & \aTop{\bullet \Dfive}{} & \aTop{\Homol_q(\Dfive;\thinspace {\mathbb{F}_5}) }{}\\
\hline  & &&\\ 
\aTop{CT(21)}{\CTtwentyone} & \ftsTwentyone & \aTop{\bullet \Dfive \bullet \Dfive}{} & 
\aTop{\left(\Homol_q(\Dfive;\thinspace {\mathbb{F}_5})\right)^2 }{}\\ 
\hline  & &&\\
\aTop{CT(24)}{\CTtwentyfour} & \ftsTwentyfour & \aTop{\bullet \Dfive \bullet \Dfive}{} & 
\aTop{\left(\Homol_q(\Dfive;\thinspace {\mathbb{F}_5})\right)^2 }{}\\ 
\hline
\end{array}$$
\caption{$5$--torsion sub-complexes of the Coxeter tetrahedral groups, in the cases where they are non-trivial.}
\label{5-torsionCT}
\end{figure}

\begin{figure}
$$\begin{array}{|c|c|l|c|c|}
\hline & & &&\\
{\rm Name} & { \aTop{\rm Coxeter}{\rm graph}}
 & \begin{array}{c} 3{\rm -torsion}\\ {\rm subcomplex \medspace quotient}\end{array}
 & \begin{array}{c}{\rm reduced  \medspace} 3{\rm -torsion}\\ {\rm subcomplex \medspace quotient}\end{array}
 & \Homol_q(CT(m);\thinspace {\mathbb{F}_3}) \\
\hline & & &&\\ 
\aTop{CT(1)}{} & \CTone & \ttsOne & 
\bullet \Dthree  & \Homol_q(\Dthree;\thinspace {\mathbb{F}_3}) \\
\aTop{CT(2)}{} & \CTtwo & \ttsTwo &
\bullet \Dthree & \Homol_q(\Dthree;\thinspace {\mathbb{F}_3}) \\
\aTop{CT(3)}{} &  \CTthree & \ttsThree & 
\bullet \Dthree  & \Homol_q(\Dthree;\thinspace {\mathbb{F}_3}) \\
\aTop{CT(7)}{} &  \CTseven & \ttsSeven &
\bullet \Dsix \bullet \Dthree & \left(\Homol_q(\Dthree;\thinspace {\mathbb{F}_3})\right)^2 \\
\aTop{CT(8)}{} &  \CTeight & \ttsEight &
\bullet \Dthree \bullet \Dthree  & \left(\Homol_q(\Dthree;\thinspace {\mathbb{F}_3})\right)^2 \\
\aTop{CT(9)}{} &  \CTnine &  \ttsNine &
\bullet \Dthree \bullet \Dthree \bullet \Dthree & \left(\Homol_q(\Dthree;\thinspace {\mathbb{F}_3})\right)^3 \\
\aTop{CT(10)}{} &  \CTten & \ttsTen   & 
\bullet \Dsix \bullet \Dthree & \left(\Homol_q(\Dthree;\thinspace {\mathbb{F}_3})\right)^2 \\
\aTop{CT(11)}{} &  \CTeleven & \ttsEleven & 
\bullet \Dsix \bullet \Dthree & \left(\Homol_q(\Dthree;\thinspace {\mathbb{F}_3})\right)^2 \\
\aTop{CT(12)}{} &  \CTtwelve & {\rm six \medspace copies \medspace of} \medspace \bullet \Dthree &
{\rm six \medspace copies \medspace of} \medspace \bullet \Dthree  & 
\left(\Homol_q(\Dthree;\thinspace {\mathbb{F}_3})\right)^6 \\
\aTop{CT(13)}{} &  \CTthirteen & \ttsThirteen &
\bullet \Dsix \bullet \Dthree \bullet \Dthree \bullet \Dthree & \left(\Homol_q(\Dthree;\thinspace {\mathbb{F}_3})\right)^4 \\
\aTop{CT(14)}{} &  \CTfourteen & \ttsFourteen &
\bullet \Dsix \bullet \Dsix \bullet \Dthree & \left(\Homol_q(\Dthree;\thinspace {\mathbb{F}_3})\right)^3  \\
\aTop{CT(15)}{} &  \CTfifteen & \ttsFifteen&
\aTop{\bullet \Dsix \bullet \Dthree \bullet \Dthree}{} & \aTop{\left(\Homol_q(\Dthree;\thinspace {\mathbb{F}_3})\right)^3}{}  \\
\aTop{CT(16)}{} &  \CTsixteen & \ttsSixteen &
\aTop{\bullet \Dthree \bullet \Dthree \bullet \Dthree}{} & \aTop{\left(\Homol_q(\Dthree;\thinspace {\mathbb{F}_3})\right)^3 }{}\\
\aTop{CT(17)}{} &  \CTseventeen &  \aTop{\bullet \Dsix \bullet \Dsix \bullet \Dthree \bullet \Dthree}{} & 
 \aTop{\bullet \Dsix \bullet \Dsix \bullet \Dthree \bullet \Dthree}{} & 
\aTop{\left(\Homol_q(\Dthree;\thinspace {\mathbb{F}_3})\right)^4}{}  \\
\aTop{CT(18)}{} & \CTeighteen & \ttsEighteen &
\aTop{\bullet \Dthree \bullet \Dthree}{} & \aTop{\left(\Homol_q(\Dthree;\thinspace {\mathbb{F}_3})\right)^2 }{}\\
\hline
\end{array}$$
\caption{$3$--torsion sub-complexes of the Coxeter tetrahedral groups $CT(1)$ through $CT(18)$, in the cases where they are non-trivial.}
\label{3-torsionCT1-14}
\end{figure}

\begin{figure}
$$\begin{array}{|c|c|l|c|c|}
\hline & & &&\\
{\rm Name} & { \aTop{\rm Coxeter}{\rm graph}}
 & \begin{array}{c} 3{\rm -torsion}\\ {\rm subcomplex \medspace quotient}\end{array}
 & \begin{array}{c}{\rm reduced  \medspace} 3{\rm -torsion}\\ {\rm subcomplex \medspace quotient}\end{array}
 & \Homol_q(CT(m);\thinspace {\mathbb{F}_3}) \\
\hline & & &&\\ 
\aTop{CT(19)}{} & \CTnineteen & \ttsNineteen &
 \aTop{\bullet \Dthree}{} & 
\aTop{\Homol_q(\Dthree;\thinspace {\mathbb{F}_3})}{}  \\ 
\aTop{CT(20)}{} & \CTtwenty & \ttsTwenty &
 \aTop{\bullet \Dthree}{} & 
\aTop{\Homol_q(\Dthree;\thinspace {\mathbb{F}_3})}{}  \\ 
\aTop{CT(21)}{} & \CTtwentyone & \ttsTwentyone &
\aTop{\bullet \Dthree \bullet \Dthree}{} & \aTop{\left(\Homol_q(\Dthree;\thinspace {\mathbb{F}_3})\right)^2 }{}\\
\aTop{CT(22)}{} & \CTtwentytwo & \ttsTwentytwo &
\aTop{\bullet \Dthree \bullet \Dthree}{} & \aTop{\left(\Homol_q(\Dthree;\thinspace {\mathbb{F}_3})\right)^2 }{}\\
\aTop{CT(23)}{} & \CTtwentythree & \ttsTwentythree &
 \aTop{\bullet \Dthree}{} & 
\aTop{\Homol_q(\Dthree;\thinspace {\mathbb{F}_3})}{}  \\
\aTop{CT(24)}{} & \CTtwentyfour & \ttsTwentyfour &
 \aTop{\bullet \Dthree}{} & 
\aTop{\Homol_q(\Dthree;\thinspace {\mathbb{F}_3})}{}  \\
\aTop{CT(25)}{} & \CTtwentyfive & \ttsTwentyfive &
 \aTop{\bullet \Dthree}{} & 
\aTop{\Homol_q(\Dthree;\thinspace {\mathbb{F}_3})}{}  \\
\aTop{CT(26)}{} & \CTtwentysix & \ttsTwentysix &
\aTop{\bullet \Dthree \bullet \Dthree}{} & \aTop{\left(\Homol_q(\Dthree;\thinspace {\mathbb{F}_3})\right)^2 }{}\\
\aTop{CT(27)}{} &  \CTtwentyseven & \ttsTwentyseven&
\aTop{\bullet \Dthree \bullet \Dthree \bullet \Dthree}{} & \aTop{\left(\Homol_q(\Dthree;\thinspace {\mathbb{F}_3})\right)^3}{}  \\
\aTop{CT(28)}{} & \CTtwentyeight & \ttsTwentyeight &
\aTop{\bullet \Dsix \bullet \Dthree}{} & \aTop{\left(\Homol_q(\Dthree;\thinspace {\mathbb{F}_3})\right)^2 }{}\\
\aTop{CT(29)}{} & \CTtwentynine & \ttsTwentynine &
\aTop{\bullet \Dsix \bullet \Dthree \bullet \Dthree}{} & \aTop{\left(\Homol_q(\Dthree;\thinspace {\mathbb{F}_3})\right)^3 }{}\\
\aTop{CT(30)}{} & \CTthirty & \ttsThirty &
\aTop{\bullet \Dsix \bullet \Dthree \bullet \Dthree}{} & \aTop{\left(\Homol_q(\Dthree;\thinspace {\mathbb{F}_3})\right)^3 }{}\\
\aTop{CT(31)}{} &  \CTthirtyone & \ttsThirtyone&
\aTop{\bullet \Dthree}{} & \aTop{\Homol_q(\Dthree;\thinspace {\mathbb{F}_3})}{}  \\
\aTop{CT(32)}{} & \CTthirtytwo & \ttsThirtytwo &
\aTop{\bullet \Dsix \bullet \Dthree}{} & \aTop{\left(\Homol_q(\Dthree;\thinspace {\mathbb{F}_3})\right)^2 }{}\\
\hline
\end{array}$$
\caption{$3$--torsion sub-complexes of the Coxeter tetrahedral groups $CT(19)$ through $CT(32)$.}
\label{3-torsionCT15-28}
\end{figure}

Consider the groups that are generated by the reflections on the four sides of a tetrahedron in hyperbolic 3-space,
 such that the images of the tetrahedron tessellate the latter.
Up to isomorphism, there are only thirty-two such groups~\cite{ElstrodtGrunewaldMennicke};
and we call them the Coxeter tetrahedral groups $CT(n)$, with $n$ running from $1$ through $32$.

\begin{proposition}
 For all prime numbers $\ell > 2$,
 the$\mod \ell$ homology of all the Coxeter tetrahedral groups is specified in the tables in Figures~\ref{3-torsionCT1-14}
 through~\ref{5-torsionCT}
 in all the cases where it is non-trivial.
\end{proposition}
\begin{proof}
 Consider the Coxeter tetrahedral group $CT(25)$, generated by the Coxeter diagram $\CTtwentyfive$. 
Then the Davis complex of $CT(25)$
 has a strict fundamental domain isomorphic to the barycentric subdivision of the hyperbolic tetrahedron
 the reflections on the sides of which generate $CT(25)$ geometrically.
A strict fundamental domain for the action on the $3$--torsion sub-complex is then the graph
$$\ttsTwentyfive$$
where the labels specify the isomorphism types of the stabilizers,
 namely the dihedral group $\Dthree$, which also stabilizes the edges,
 the symmetric group $\Sfour$ and the semi-direct product $(\Z/2)^2 \rtimes \Sthree$.
The $\ell$--torsion sub-complexes for all greater primes $\ell$ are empty.
By Theorem~\ref{small rank Coxeter groups}, we can reduce the $3$--torsion sub-complex to a single vertex and obtain 
$\Homol_*(CT(25);\thinspace {\mathbb{F}_3}) \cong \Homol_*(\Dthree;\thinspace {\mathbb{F}_3}).$
For the other Coxeter tetrahedral groups, we proceed analogously.
\end{proof}
\mbox{The entries in the tables in Figures \ref{3-torsionCT1-14}--\ref{5-torsionCT}
 have additionally been checked on the machine \cite{HAP}.}

\bigskip
\bigskip

\section{The reduced torsion sub-complexes of the Bianchi groups} \label{The conjugacy classes of finite order elements}

The groups $\mathrm{SL_2}(\ringO_{-m})$ act in a natural way on real hyperbolic three-space~$\Hy$,
 which is isomorphic to the symmetric space $\mathrm{SL_2}(\C)/\mathrm{SU}(2)$  associated to them.
The kernel of this action is the center $\{ \pm 1 \}$ of the groups.
Thus it is useful to study the quotient of $\mathrm{SL_2}(\ringO_{-m})$ by its center, namely $\mathrm{PSL_2}(\ringO_{-m})$,
which we also call a Bianchi group.
Let~$\Gamma$ be a finite index subgroup in $\text{PSL}_2(\mathcal{O}_{-m})$. 
Then any element of~$\Gamma$ fixing a point inside~$\Hy$ acts as a rotation of finite order.
By Klein, we know conversely that any torsion element~$\alpha$ is elliptic and hence fixes some geodesic line.
We call this line \emph{the rotation axis of~$\alpha$}.
Every torsion element acts as the stabilizer of a line conjugate to one passing through the Bianchi fundamental polyhedron.
Let~$X$ be the \textit{refined cellular complex} obtained from the action of~$\Gamma$ on~$\Hy$
 as described in~\cite{Rahm_homological_torsion}, 
namely we subdivide~$\Hy$  until the stabilizer in~$\Gamma$ of any cell $\sigma$ fixes $\sigma$ point-wise. 
We achieve this by computing Bianchi's fundamental polyhedron for the action of~$\Gamma$,
 taking as preliminary set of 2-cells its facets lying on the Euclidean hemispheres
 and vertical planes of the upper-half space model for $\Hy$,
 and then subdividing along the rotation axes of the elements of~$\Gamma$. 

It is well-known that if $\gamma$ is an element of finite order $n$ in a Bianchi group, then $n$ must be 1, 2, 3, 4 or 6,
 because $\gamma$ has eigenvalues $\rho$ and $\overline{\rho}$,
 with $\rho$ a primitive $n$--th root of unity, and the trace of~$\gamma$ is $\rho + \overline{\rho} \in \ringO_{-m} \cap \R = \Z$.
For $\ell$ being one of the two occurring prime numbers $2$ and~$3$, the orbit space of this sub-complex is a finite graph,
 because the cells of dimension greater \mbox{than 1} are trivially stabilized in the refined cellular complex.

For the Bianchi groups, we can see how to construct the reduced torsion sub-complex outside of the geometric model,
by constructing the following conjugacy classes graphs.
Let $\ell$ be a prime number.
For a circle to become a graph, we identify the two endpoints of a single edge.

\vbox{
\begin{definition}
The $\ell$--\emph{conjugacy classes graph} of an arbitrary group~$\Gamma$ is given by the following construction.
\end{definition}
\begin{itemize}
 \item We take as vertices the conjugacy classes of finite subgroups $G$ of~$\Gamma$ containing elements $\gamma$ of order $\ell$ such that the normalizer of $\langle \gamma \rangle$ in $G$ is not $\langle \gamma \rangle$ itself.
\item We connect two vertices by an edge if and only if they admit representatives sharing a common subgroup of order $\ell$.
\item For every pair of subgroups of order $\ell$ in $G$, which are conjugate in~$\Gamma$ but not in~$G$, we draw a circle attached to the vertex labeled by~$G$.
\item For every conjugacy class of subgroups of order $\ell$ which are not properly contained in any finite subgroup of~$\Gamma$, we add a disjoint circle.
\end{itemize}
}

\begin{theorem} \label{isomorphy of graphs}
Let~$\Gamma$ be a finite index subgroup in a Bianchi group with units $\{\pm 1\}$ and $\ell$ any prime number.
 Then the $\ell$--\emph{conjugacy classes graph}
 and the \emph{quotient of the reduced $\ell$--torsion sub-complex} of the action of~$\Gamma$ on hyperbolic $3$--space
 are isomorphic graphs.
\end{theorem}

The remainder of this section will be devoted to the proof of this theorem.
The first ingredient is the following classification of Felix Klein \cite{binaereFormenMathAnn9}.
\begin{Lem}[Klein] \label{finiteSubgroups}
The finite subgroups in $\mathrm{PSL}_2(\ringO)$
 are exclusively of isomorphism types the cyclic groups of orders one, two and three,
 the Klein four-group \mbox{$\Kleinfourgroup \cong \Z/2 \times \Z/2$},
 the dihedral group $\Sthree$ with six elements (non-commutative) and the alternating group~$\Afour$.
\end{Lem}

The proof of the following lemma from \cite{Rahm_homological_torsion} 
passes unchanged from $\mathrm{PSL}_2(\ringO)$ to any of its finite index subgroups $\Gamma$.

\begin{lemma} \label{geometricRigiditytheorem}
Let $v$ be a non-singular vertex in the refined cell complex.
 Then the number~$\bf n$ of orbits of edges adjacent to $v$ in the refined cellular complex $X$,
 with stabilizer in $\Gamma$ isomorphic to~$\Z/ \ell$,  is given as follows for $\ell = 2$ and $\ell = 3$.
$$ \begin{array}{c|cccccc}
{\rm Isomorphism}\medspace { \rm type}\medspace {\rm of } \medspace {\rm the  }\medspace {\rm vertex }\medspace {\rm stabiliser} & \{1\} & \Z/2 & \Z/3 & \Kleinfourgroup & \Sthree & \Afour \\
\hline &&&&&&  \\
{\bf n} \medspace \mathrm{ for } \medspace \ell = 2 & 0 & 2 & 0 & 3 & 2 & 1 \\ 
{\bf n} \medspace \mathrm{ for } \medspace \ell = 3 & 0 & 0 & 2 & 0 & 1 & 2.
\end{array} $$ \normalsize
\end{lemma}

Alternatively to the case-by-case proof of \cite{Rahm_homological_torsion}, we can proceed by investigating
the action of the associated normalizer groups. 
Straight-forward verification using the multiplication tables of the concerned finite groups yields the following.

Let $G$ be a finite subgroup of ${\rm PSL}_2(\ringO_{-m})$. Then the type of the normalizer of any subgroup of type~$\Z/ \ell$ in $G$ is given as follows for $\ell = 2$ and $\ell = 3$, where we print only cases with existing subgroup of type $\Z/\ell$.
$$ \begin{array}{c|cccccc}
{\rm Isomorphism}\medspace { \rm type}\medspace {\rm of } \medspace G & \{1\} & \Z/2 & \Z/3 & \Kleinfourgroup & \Sthree & \Afour \\
\hline &&&&&&  \\
\mathrm{ normaliser } \medspace \mathrm{ of } \medspace \Z/2 &  &
 \Z/2 &      & \Kleinfourgroup & \Z/2    & \Kleinfourgroup \\
\mathrm{ normaliser } \medspace \mathrm{ of } \medspace \Z/3 &  &
      & \Z/3 &                 & \Sthree & \Z/3.
\end{array} $$ \normalsize

The final ingredient in the proof of Theorem~\ref{isomorphy of graphs} is the following.

\begin{lemma} \label{conjugacy class correspondence}
 There is a natural bijection between conjugacy classes of subgroups of $\Gamma$ of order~$\ell$
 and edges of the quotient of the reduced $\ell$--torsion sub-complex.
It is given by considering the stabilizer of a representative edge in the refined cell complex.
\end{lemma}

In order to prove the latter lemma, we need another lemma, and we establish it now.

\begin{remark} \label{chain of edges}
 Any edge of the reduced torsion sub-complex is obtained by merging a chain of edges on the intersection of one geodesic line with
 some strict fundamental domain for~$\Gamma$ in $\Hy$.
\end{remark}
We call this chain the \emph{chain of edges associated to $\alpha$}.
It is well defined up to translation along the rotation axis of~$\alpha$.

\begin{lemma} \label{torsion axis}
Let $\alpha$ be any non-trivial torsion element in a finite index subgroup~$\Gamma$ in a Bianchi group.
Then the $\Gamma$--image of the chain of edges associated to $\alpha$ contains the rotation axis of $\alpha$.
\end{lemma}
\begin{proof}
 Because of the existence of a fundamental polyhedron for the action of $\Gamma$ on $\Hy$, 
 the rotation axis of $\alpha$ is cellularly subdivided into compact edges such that the union over the
 $\Gamma$--orbits of finitely many of them contains all of them.

The case \circlegraph . Assume that $\langle \alpha \rangle \cong \Z/\ell$ 
is not contained in any subgroup of $\Gamma$ of type $\Dl$.
Because the inclusion $\Z/2 \hookrightarrow \Sthree$, respectively $\Z/3 \hookrightarrow \Afour$, 
induces an isomorphism on mod $2$, respectively mod $3$, homology,
we can merge those edges orbit-wise until the neighbouring edges are on the same orbit.
So the reduced edge admits a $\Gamma$--image containing the rotation axis of $\alpha$.

The case \edgegraph . Make the complementary assumption that there is a subgroup of $\Gamma$ of type~$\Dl$,
containing $\langle \alpha \rangle \cong \Z/\ell$.
Then that subgroup contains a reflection $\beta$ of the rotation axis of $\alpha$ onto itself at a vertex $v$ stabilized by $\Dl$,
or by $\Afour \supset \Kleinfourgroup$.
Then by Lemma~\ref{geometricRigiditytheorem}, 
the $\Gamma$--orbits of the edges on the rotation axis of $\alpha$ cannot close into a loop \circlegraph .
So at the other end of the reduced edge $e$ originating at $v$, there must be another vertex of stabilizer $\Dl$,
respectively $\Afour \supset \Kleinfourgroup$, containing a second reflection $\gamma$ of the rotation axis of $\alpha$.
The latter reflection turns the axis as illustrated by the following images of~$e$ : \illustration .
The images of the reduced edge under the words in $\beta$ and $\gamma$ tessellate the whole rotation axis of $\alpha$.
\end{proof}

\begin{proof}[Proof of Lemma \emph{\ref{conjugacy class correspondence}}]
Consider a subgroup $\langle \alpha \rangle \cong \Z/\ell$ of $\Gamma$.
We need to study the effect of conjugating it by an element $\gamma \in \Gamma$.
Obviously, $\alpha$ and $\gamma \alpha  \gamma^{-1}$ stabilize edges on the same $\Gamma$--orbit.

One immediately checks that any fixed point $x \in \Hy$ of $\alpha$ induces the fixed point~$\gamma(x)$ of~$\gamma \alpha \gamma^{-1}$. As PSL$_2(\C)$ acts by isometries, the whole fixed point sets are identified.
Hence the fixed point set in $\Hy$ of $\alpha$ is identified by~$\gamma$ with the fixed point set of~$\gamma \alpha \gamma^{-1}$.
Therefore, we know that the line fixed by~$\alpha$ is sent by~$\gamma$ to the line fixed by~$\gamma \alpha  \gamma^{-1}$.

By Lemma \ref{torsion axis}, the union of the $\Gamma$--images of the chain associated to~$\alpha$
 contains the whole geodesic line fixed by $\alpha$.
As the $\Gamma$--action is cellular,
 any cell stabilized by~$\gamma \alpha  \gamma^{-1}$ admits a cell on its orbit stabilized by $\alpha$.
So it follows that precisely the edges stabilized by the elements of the conjugacy class of $\langle \alpha \rangle$
 pass to the reduced edge orbit obtained from the chain of edges associated to $\alpha$.
\end{proof}

\begin{proof} [Proof of Theorem \emph{\ref{isomorphy of graphs}}]
 Comparing with Lemma~\ref{geometricRigiditytheorem},
 we see that the vertex set of the $\ell$--conjugacy classes graph gives precisely the bifurcation points and vertices
 with only one adjacent edge of the orbit space of the $\ell$--torsion sub-complex.
When passing to the orbit space of the \emph{reduced} $\ell$--torsion sub-complex,
 we get rid of all vertices with two adjacent edges.
The disjoint circles~\circlegraph that we can obtain in the orbit space look like an exception,
but in fact there is just one adjacent edge, touching the vertex from both sides.
 By Lemma~\ref{conjugacy class correspondence},
 the edges of the $\ell$--conjugacy classes graph give the edges of the quotient of the reduced $\ell$--torsion sub-complex.
\end{proof}

\bigskip

\section{The Farrell cohomology of the Bianchi groups} \label{The Kraemer numbers and group homology}
In this section, we are going to prove Theorem~\ref{Grunewald-Poincare series formulae}.
In order to compare with Kr\"amer's formulae that we evaluate in the Appendix,
 we make use of his notations for the numbers of conjugacy classes
 of the five types of non-trivial finite subgroups in the Bianchi groups.
We apply this also to the conjugacy classes in the finite index subgroups in the Bianchi groups.
Kr\"amer's symbols for these numbers are printed in the first row of the below table,
 and the second row gives the symbol for the type of counted subgroup.
$$\begin{array}{|c|c|c|c|c|c|c|c|}
\hline &&&&&&& \\
   \mu_2           & \mu_T  & \mu_3   & \lambda_{2\ell} & \lambda_4^T                  & \lambda_4^*                                     & \lambda_6^*                   & \mu_2^- \\
\hline &&&&&&& \\
   \Kleinfourgroup & \Afour & \Sthree & \Z/\ell         & \Z/2 \subset \Afour  & \Z/2 \subset \Kleinfourgroup & \Z/3 \subset \Sthree  & \Kleinfourgroup \nsubseteq \Afour \\
\hline
  \end{array}$$
Here, the inclusion signs ``$\subset$'' mean that we only consider copies of $\Z/\ell$ admitting the specified inclusion in the given Bianchi group and $\Kleinfourgroup \nsubseteq \Afour$ means that we only consider copies of $\Kleinfourgroup$ not admitting any inclusion into a subgroup of type $\Afour$ of the Bianchi group.

Note that the number $\mu_2^-$ is simply the difference $\mu_2 -\mu_T$, because every copy of $\Afour$ admits precisely one normal subgroup of type~$\Kleinfourgroup$.
Also, note the following graph-theoretical properties of the quotient of the reduced torsion subcomplex, the latter of which we obtain by restricting our attention to the connected components not homeomorphic to~$\circlegraph$.

\begin{corollary}[Corollary to Lemma~\ref{geometricRigiditytheorem}] \label{graph-theoretic formulae}
For all finite index subgroups in Bianchi groups with units $\{\pm 1\}$, the numbers of conjugacy classes of finite subgroups  satisfy
 $\lambda_4^T \leq \mu_T$ and $2\lambda_6^* = \mu_3$, and even
$$2\lambda_4^* = \mu_T +3 \mu_2^- .$$
\end{corollary}

The values given by Kr\"amer's formulae are matching with the values computed with~\cite{BianchiGP}.

\begin{observation} \label{Kraemer numbers determine homological 3-torsion}
The numbers of conjugacy classes of finite subgroups determine the 3-conjugacy classes graph and hence the quotient of the reduced 
$3$--torsion sub-complex for all finite index subgroups in Bianchi groups with units $\{\pm 1\}$,
 as we can see immediately from Theorem~\ref{isomorphy of graphs} and the description of the reduced
 $3$--torsion sub-complex in \cite{Rahm_homological_torsion}.
\end{observation}

For the $2$--torsion part of the proof of Theorem~\ref{Grunewald-Poincare series formulae}, we still need the following supplementary ingredients.

\begin{remark} \label{directSumDecomposition}
In the equivariant spectral spectral sequence converging to the Farrell cohomology of a given finite index subgroup $\Gamma$ in $\PSLO$,
the restriction of the differential to maps between cohomology groups of cells that are not adjacent in the orbit space,
 are zero. 
So, the $\ell$--primary part of the degree--$1$--differentials of this sequence can be decomposed
 as a direct sum of the blocks associated to the connected components of the quotient of the $\ell$--torsion sub-complex
(Compare with sub-lemma 45 of~\cite{Rahm_homological_torsion}).
\end{remark}

\begin{Lem}[Schwermer/Vogtmann] \label{inducedMaps}
Let $M$ be $\Z$ or $\Z/2$. Consider group homology with trivial $M$--coefficients. Then the following holds.
\begin{itemize}
\item Any inclusion $\Z/2 \to \Sthree$ induces an injection on homology.
\item An inclusion $\Z/3 \to \Sthree$ induces an injection on homology in degrees congruent to $3$ or $0 \mod 4$, and is otherwise zero.
\item Any inclusion $\Z/2 \to \Kleinfourgroup$ induces an injection on homology in all degrees.
\item An inclusion $\Z/3 \to \Afour$ induces injections on homology in all degrees.
\item An inclusion $\Z/2 \to \Afour$ induces injections on homology in degrees greater than~$1$, and is zero on~$\Homol_1$.
\end{itemize}
\end{Lem}

For the proof in $\Z$--coefficients, see \cite{SchwermerVogtmann}, for $\Z/2$--coefficients see \cite{Rahm_homological_torsion}.

\begin{Lem}[\cite{Rahm_homological_torsion}, lemma 32] \label{D2blocks}
Let $q \geq 3$ be an odd integer number. Let $v$ be a vertex representative of stabilizer type $\Kleinfourgroup$ in the refined cellular complex for the Bianchi groups. Then the three images in $\left(\Homol_q(\Kleinfourgroup;\Z) \right)_{(2)}$ induced by the inclusions of the stabilizers of the edges adjacent to $v$, are linearly independent.
\end{Lem}

Finally, we establish the following last ingredient for the proof of Theorem~\ref{Grunewald-Poincare series formulae},
 which might be of interest in its own right.
Let $\Gamma$ be a finite index subgroup in a Bianchi group, and consider its action on the refined cellular complex.

\begin{lemma} \label{injectivity}
In all rows $q > 1$ and outside connected components of quotient type $\circlegraph$,
 the $2$--torsion part of the $d^1_{p,q}$--differential of the equivariant spectral sequence converging to 
 $\Homol_{p+q}(\Gamma; \thinspace \Z)$ is always injective.
\end{lemma}
\begin{proof}
 For matrix blocks of the $2$--torsion part of the $d^1_{p,q}$--differential associated to vertices with just one adjacent edge,
 we see from Lemma~\ref{geometricRigiditytheorem} that the vertex stabilizer is of type~$\Afour$ in \mbox{$2$--torsion,}
 so injectivity follows from Lemma~\ref{inducedMaps}.
As we have placed ourselves outside connected components of quotient type $\circlegraph$,
 the remaining vertices are bifurcation points of stabilizer type~$\Kleinfourgroup$ and injectivity follows from Lemma~\ref{D2blocks}.
\end{proof}

\begin{proof}[Proof of Theorem \emph{\ref{Grunewald-Poincare series formulae}}.]
In $3$--torsion, Theorem~\ref{Grunewald-Poincare series formulae}
 follows directly from Observation~\ref{Kraemer numbers determine homological 3-torsion},
 Corollary~\ref{graph-theoretic formulae} and Theorem~\ref{pivotal}.
In $2$--torsion, what we need to determine with the numbers of conjugacy classes of finite subgroups, is the
\mbox{$2$--primary} part of the $E^2_{p,q}$--term of the equivariant spectral sequence converging to 
 $\Homol_{p+q}(\Gamma; \thinspace \Z)$ in all rows $q > 1$.
 From there, we see from Theorem~\ref{pivotal} that we obtain the claim.
By Remark~\ref{directSumDecomposition},
 we only need to check this determination on each homeomorphism type of connected components of the quotient of the reduced $2$--torsion subcomplex.
 We use Theorem~\ref{isomorphy of graphs} to identify the quotient of the reduced $2$--torsion subcomplex and the $2$--conjugacy classes graph.
 Then we can observe that
\begin{itemize}
 \item Kr\"amer's number $\lambda_4^* -\lambda_4$ determines the number of connected components of type $\circlegraph$.
\item Kr\"amer's number $\lambda_4^*$ determines the number of edges of the $2$--torsion subcomplex orbit space outside connected components of type $\circlegraph$.
Lemma~\ref{injectivity} tells us that the block of the $d^1_{p,q}$--differential of the equivariant spectral sequence associated to such edges is always injective.
\item Kr\"amer's number $\mu_2^-$ determines the number of bifurcation points, and $\mu_T$ determines the number of vertices with only one adjacent edge of the $2$--torsion subcomplex orbit space.
\end{itemize}
Using Corollary~\ref{graph-theoretic formulae}, we obtain the explicit formulae in Theorem~\ref{Grunewald-Poincare series formulae}.
\end{proof}

\bigskip

\section{The cohomology ring structure of the Bianchi groups} \label{cohomology ring}

In \cite{BerkoveMod2}, Berkove has found a compatibility of the cup product of the cohomology ring of a Bianchi group
 with the cup product of the cohomology rings of its finite subgroups.
 This compatibility within the equivariant spectral sequence implies that all products that come
 from different connected components of the quotient of the reduced torsion sub-complex
 (which we turn into the conjugacy classes graph in Section~\ref{The conjugacy classes of finite order elements}) are zero.
It follows that the cohomology ring of any Bianchi group splits into a restricted sum over sub-rings,
 which depend in degrees above the virtual cohomological dimension only on the homeomorphism type
 of the associated connected component of the quotient of the reduced torsion sub-complex.
The analogue in cohomology of Theorem~\ref{Grunewald-Poincare series formulae} and Berkove's computations of sample cohomology rings~\cite{Berkove}
 yield the following corollary in $3$--torsion.

We use Berkove's notation, in which the degree $j$ of a cohomology generator $x_j$ is appended as a subscript.
Furthermore, writing cohomology classes inside square brackets means that they are polynomial (of infinite multiplicative order), and writing them inside parentheses means that they are exterior (their powers vanish).
The restricted sum $\widetilde{\oplus}$ identifies all the degree zero classes into a single copy of~$\Z$; when we write it with a power, we specify the number of summands. Recall that $\lambda_6$ (respectively~$\mu_3$) counts the number of conjugacy classes of subgroups of type $\Z/3$ (respectively~$\Sthree$) in the Bianchi group.

\begin{corollary}
 In degrees above the virtual cohomological dimension, the $3$--primary part of the cohomology ring of any Bianchi group~$\Gamma$ with units $\{ \pm 1 \}$ is given by 
$$ \Homol^*(\Gamma ; \thinspace \Z)_{(3)} 
\cong \widetilde{\oplus}^{(\lambda_6 -\frac{\mu_3}{2})} \Z[x_2](\sigma_1)
\medspace \widetilde{\oplus}^\frac{\mu_3}{2} \Z[x_4](x_3),  $$
where the generators $x_j$ are of additive order $3$.
\end{corollary}
In $2$--torsion, it does in general not suffice to know only the numbers of conjugacy classes of finite subgroups
 to obtain the cohomology ring structure,
 because for the two reduced $2$--torsion sub-complex orbit spaces $\graphTwo \graphTwo$ and $\graphFive \edgegraph$,
 we obtain the same numbers of conjugacy classes and homological $2$--torsion,
 but different multiplicative structures of the mod-$2$ cohomology rings, as we can see from Table~\ref{Restricted summands},
 which we compile from the results of~\cite{BerkoveMod2} (and~\cite{Rahm_homological_torsion}).
\begin{table}
$$\begin{array}{|c|c|}
\hline
T & {\rm Subring \medspace associated \medspace to \medspace connected \medspace components \medspace of \medspace type \medspace } T {\rm \medspace in \medspace the \medspace 2\text{--conjugacy} \medspace classes \medspace graph} \\
\hline & \\ 
\circlegraph & \F_2[n_1](m_1) \\ &\\
\edgegraph & \F_2[m_3, u_2, v_3, w_3]/
\langle m_3 v_3 = 0, \quad u_2^3 +w_3^2 +v_3^2 +m_3^2 +w_3(v_3 +m_3) = 0 \rangle \\ &\\
\graphTwo & \F_2[n_1, m_2, n_3, m_3]/ 
\langle n_1 n_3  = 0,  \quad m_2^3 +m_3^2 +n_3^2 +m_3 n_3 +n_1 m_2 m_3 = 0 \rangle \\& \\
\graphFive & \F_2[n_1, m_1, m_3]/
\langle m_3(m_3 +n_1^2 m_1 +n_1 m_1^2) = 0 \rangle \\ &\\
\hline
\end{array}$$
\caption{Restricted summands of  $\Homol^*(\PSLO; \F_2)$ above the virtual cohomological dimension.}
\label{Restricted summands} 
\end{table}
\begin{observation}
In the cases of class numbers $1$ and $2$, only the homeomorphism
types $T$ listed in Table~\ref{Restricted summands} occur as connected components in the quotient of the reduced $2$--torsion sub-complex.
So for all such Bianchi groups with units $\{ \pm 1 \}$, the mod-2 cohomology ring $\Homol^*(\PSLO; \F_2)$ splits, above the virtual cohomological dimension,
 as a restricted sum over the sub-rings specified in Table~\ref{Restricted summands}, with powers according to the multiplicities of the occurrences of the types $T$.
\end{observation}

\bigskip
\newpage

\begin{appendix}
\section{Numerical evaluation of Kr\"amer's formulae}
\subsection{Numbers of conjugacy classes in $3$--torsion} \label{Numerical evaluation of Kraemer's formulae in 3-torsion}

Denote by $\delta$ the number of finite ramification places of $\rationals(\sqrt{-m}\thinspace)$ over $\rationals$.
Let $k_+$ be the totally real number field $\rationals(\sqrt{3m}\thinspace)$ and denote its ideal class number by $h_{k_+}$.
Kr\"amer introduces the following indicators:
$$ z := \begin{cases} 2, & {\rm if} \medspace 3 \medspace {\rm is} \medspace {\rm the} \medspace {\rm norm} \medspace {\rm of} \medspace {\rm an} \medspace {\rm integer} \medspace {\rm of} \medspace k_+, \\
1,  & {\rm otherwise.} \end{cases}$$
For $m \equiv 0 \mod 3$ and $m \neq 3$, denote by $\epsilon := \frac{1}{2}(a+b\sqrt{\frac{m}{3}}) > 1$ the fundamental unit of~$k_+$ (where $a, b \in \N$).
Now, define
$$ x' := \begin{cases} 2, & {\rm if} \medspace {\rm the} \medspace {\rm norm} \medspace {\rm of} \medspace \epsilon \medspace {\rm is} \medspace 1, \\
1,  & {\rm if} \medspace {\rm the} \medspace {\rm norm} \medspace {\rm of} \medspace \epsilon \medspace {\rm is} \medspace -1 \end{cases}$$
and
$$ y := \begin{cases} 2, & {\rm if} \medspace b \equiv 0 \mod 3, \\
1,  & {\rm otherwise.} \end{cases}$$
Then \cite{Kraemer}*{20.39 and 20.41} yield the following formulae in $3$--torsion.
$$\begin{array}{|l|l|c|}
\hline & &\\
 m {\rm \medspace specifying \medspace Bianchi \medspace groups } \medspace \PSLO  & \lambda_6^* & \lambda_6 -\lambda_6^* \\
\hline & &\\
m \equiv 2 \mod 3 & 0 & \frac{z}{2}h_{k_+} \\
& &\\
\hline & &\\
m \equiv 1 \mod 3 {\rm \medspace gives \medspace either }& 2^{\delta -1} & \frac{1}{2}(h_{k_+} -2^{\delta -1}) \\
& &\\
 {\rm \medspace or }& 0 & \frac{1}{2}h_{k_+} \\
& &\\
\hline & &\\
m \equiv 6 \mod 9 & 0 & x' y h_{k_+} \\
& &\\
\hline & &\\
m \equiv 3 \mod 9, \medspace m \neq 3 {\rm \medspace gives \medspace either }& 2^{\delta -2} & \frac{1}{2}(3x' h_{k_+} -2^{\delta -2}) \\
& &\\
 {\rm \medspace or }& 0 & \frac{1}{2}3 x' h_{k_+} \\
& &\\
\hline
\end{array}$$
The above case distinctions come from the fact that Kr\"amer's theorem 20.39 ranges over all types of maximal orders in quaternion algebras over $\rationals(\sqrt{-m}\thinspace)$, in which Kr\"amer determines the numbers of conjugacy classes in the norm-1-group.
The remaining task in order to decide which of the cases applies, is to find out of which type considered in the mentioned theorem is the maximal order M$_2(\ringO_{-m})$.
Some methods to cope with this task are introduced in \cite{Kraemer}*{\S 27}.

\newpage
Kr\"amer's resulting criteria can be summarized as follows for 3-torsion.

$$\begin{array}{|l|l|}
\hline & \\
   {\rm condition} & {\rm implication} \\
\hline & \\
m \equiv 2 \mod 3 & \mu_3 = \lambda_6^* = 0. \\& \\
m \equiv 6 \mod 9 & \mu_3 = \lambda_6^* = 0. \\& \\
m \medspace {\rm prime} \medspace {\rm and} \medspace m \equiv 1   \mod 3 & \lambda^*_6 > 0. \\& \\
m = 3p \medspace {\rm with} \medspace p \medspace {\rm prime} \medspace {\rm and} \medspace p \equiv 1  \mod 3 & \lambda^*_6 > 0. \\& \\

m \equiv 1 \mod 3  \medspace {\rm and} \medspace -3 \medspace {\rm occurs} \medspace {\rm as} \medspace {\rm norm}  \medspace {\rm on} \medspace \ringO_{k_+}
&  \lambda^*_6 > 0. \\& \\

m \equiv 1 \mod 3  \medspace {\rm and} \medspace -3 \medspace {\rm does} \medspace {\rm not} \medspace {\rm occur} \medspace {\rm as} \medspace {\rm norm}  \medspace {\rm on} \medspace \ringO_{k_+}
& \lambda_6 -\lambda^*_6 > 0. \\& \\

m \equiv 1 \mod 3  \medspace {\rm and} \medspace m \medspace {\rm admits} \medspace {\rm a} \medspace {\rm prime} \medspace {\rm divisor} \medspace p  \medspace {\rm with} \medspace p \equiv 2 \mod 3
& \lambda_6 -\lambda^*_6 > 0. \\& \\

m \equiv 3  \mod 9 \medspace {\rm and} \medspace x' = 1 &\\ \medspace {\rm and}  \medspace  \frac{m}{3} \medspace {\rm admits} \medspace {\rm only} \medspace {\rm prime}  \medspace {\rm divisors} \medspace p \medspace {\rm with} \medspace p \equiv 1 \mod 12 & \lambda^*_6 > 0. \\& \\
m \equiv 3  \mod 9 \medspace {\rm and} \medspace x' = 1 &\\ \medspace {\rm and}  \medspace  \frac{m}{3} \medspace {\rm admits} \medspace {\rm a} \medspace {\rm prime}  \medspace {\rm divisor} \medspace p \medspace {\rm with} \medspace p \equiv 5 \mod 12 & \lambda^*_6 = 0. \\& \\
m \equiv 3  \mod 9 \medspace {\rm and} \medspace  h(k_+') = 2^{\delta-3} &\\ \medspace {\rm and}  \medspace  \frac{m}{3} \medspace {\rm admits} \medspace {\rm only} \medspace {\rm prime}  \medspace {\rm divisors} \medspace p \medspace {\rm with} \medspace p \equiv \pm 1 \mod 12 \medspace {\rm or} \medspace p = 2 & \lambda^*_6 = 0. \\& \\
m \equiv 3  \mod 9 \medspace {\rm and} \medspace  h(k_+') = 1 &\\ \medspace {\rm and}  \medspace  \frac{m}{3} = p'p \medspace {\rm with} \medspace p',p \medspace {\rm prime}  \medspace {\rm and} \medspace p' \equiv p  \equiv 7 \mod 12  & \lambda^*_6 > 0.  \\& \\
\hline 
  \end{array}$$

In order to determine Kr\"amer's indicator $z$, we need to determine if a given value occurs as the norm on the ring of integers of an imaginary quadratic number field. This is implemented in Pari/GP \cite{Pari} (the first step is computing the answer under the Generalized Riemann hypothesis, and the second step is a check computation which confirms that we arrive at that answer without this hypothesis). Additionally, we compare with the below criterion \cite{Kraemer}*{(20.13)}.
\begin{lemma}[Kr\"amer]
Let $m$ be not divisible by $3$.
\begin{itemize}
\item If the number $-3$ is the norm of an integer in the totally real number field $k_+$, then all prime divisors $p\in \N$ of $m$ satisfy the congruence $p \equiv 1 \mod 3$. \\
 Especially, the congruence $m \equiv 1 \mod 3$ is implied.
\item If the number $3$ is the norm of an integer in the totally real number field $k_+$, then all prime divisors $p\in \N$ of $m$ satisfy either $p = 2$ or the congruence $p \equiv \pm 1 \mod 12$. \\
Additionally, the congruence $m \equiv 2 \mod 3$ is implied.
\end{itemize}
\end{lemma}

\vbox{
With Kr\"amer's criteria at hand, we can decide for many Bianchi groups, which of the alternative cases in Kr\"amer's formulae must be used.
We do this in the below tables for all such Bianchi groups $\PSLO$ with absolute value of the discriminant $\Delta$ ranging between $7$ and $2003$, where we recall that the discriminant is \mbox{$\Delta = \begin{cases} -m, & m \equiv 3 \mod 4, \\
-4m,  & \mathrm{else.} \end{cases}$}

In the cases $m \in \{102, 133, 165, 259, 559, 595, 763, 835, 1435\}$, where these statements are not sufficient to eliminate the wrong alternatives, we insert the results of~\cite{BianchiGP}. This way, the below tables treat all Bianchi groups with units $\{\pm 1\}$ and discriminant of absolute value less than $615$. The cases where an ambiguity remains (so to exclude them from our tables) are
$m \in \{210, 262, 273, 298, 345, 426, 430, 462, 481, 615, 1155, 1159, 1195, 1339, 1351, 1407, 1515, 1807\}$.
For tables of the cases without ambiguity, with $m$ ranging up to 10000, see the preprint version~2 of this paper on HAL.

In~\cite{manuscript}, a theorem is established which solves all these ambiguities by giving for each type of finite subgroups in $\PSLO$ criteria equivalent to its occurrence, in terms of congruence conditions on the prime divisors of $m$. 
\bigskip

\begin{tabular}{|l|l|}
\hline & \\
 \scriptsize  $\begin{array}{c}3{\rm -conjugacy} \\ {\rm classes  \medspace graph}\end{array}$ & \normalsize $m {\rm \medspace specifying \medspace \medspace Bianchi \medspace groups} \medspace \PSLO \medspace {\rm with \medspace this} \medspace 3{\rm -conjugacy \medspace classes  \medspace graph}$ \\
\hline & \\
$\circlegraph $  & \scriptsize 2, 5, 6, 10, 11, 14, 15, 17, 22, 23, 29, 34, 35, 38, 41, 46, 47, 51, 53, 55, \\ & \scriptsize 58, 59, 62, 71, 82, 83, 86, 87, 89, 94, 95, 101, 106, 113, 115, 118, 119, 123, 131, 134, \\ & \scriptsize 137, 142, 149, 155, 158, 159, 166, 167, 173, 178, 179, 187, 191, 197, 202, 203, 206, 214, 215, 226, \\ & \scriptsize 227, 233, 235, 239, 251, 254, 257, 263, 267, 269, 274, 278, 281, 287, 293, 295, 303, 311, 317, 319, \\ & \scriptsize 323, 326, 334, 335, 339, 346, 347, 353, 355, 358, 359, 371, 382, 383, 389, 391, 394, 395, 398, 401, \\ & \scriptsize 411, 415, 422, 431, 443, 446, 447, 449, 451, 454, 461, 466, 467, 478, 479, 491, 515, 519, 527, 535, \\ & \scriptsize 551, 563, 583, 591, 599, 623, 635, 647, 655, 659, 667, 683, 695, 699, 707, 719, 731, 743, 755, 779, \\ & \scriptsize 791, 799, 807, 815, 827, 839, 843, 879, 887, 895, 899, 911, 943, 947, 951, 955, 959, 979, 983, 995, \\ & \scriptsize 1003, 1019, 1031, 1055, 1059, 1091, 1103, 1111, 1115, 1135, 1139, 1151, 1163, 1167, 1187, 1207, \\ & \scriptsize 1211, 1219, 1223, 1243,  1247, 1255, 1259, 1271, 1283, 1307, 1315, 1343, 1347, 1363, 1367, 1379, \\ & \scriptsize 1383, 1411, 1415, 1439, 1487, 1499, 1507, 1511,  1523, 1527, 1535, 1555, 1559, 1563, 1571, 1607, \\ & \scriptsize 1631, 1639, 1643, 1655, 1667, 1671, 1707, 1711, 1735, 1751, 1763, 1779,  1787, 1795, 1799, 1811, \\ & \scriptsize 1819, 1823, 1835, 1847, 1851, 1883, 1903, 1907, 1915, 1919, 1923, 1927, 1931, 1943, 1959, 1979,  2003,  \\ & \\ 
2 $\circlegraph $  & \scriptsize 26, 42, 65, 69, 70, 74, 77, 78, 85, 110, 122, 130, 141, 143, 145, 154, 161, 170, 182, 185, 186, 190, 194, \\ & \scriptsize  195, 205, 209, 213, 218, 221, 222, 230, 231, 238, 253, 265,  266, 286, 305, 310, 314, 322, 329, 365,  \\ & \scriptsize 366, 370, 377, 386, 406, 407, 410, 418, 434, 437, 442, 445, 455,  458, 470, 473,  474, 483, 485, 493, 494, \\ & \scriptsize 497, 555, 611, 627, 671, 715, 767, 803, 851, 923, 935, 1015, 1079, 1095, 1199,  1235, 1295, 1311, 1391,  \\ & \scriptsize  1403, 1455, 1463, 1491, 1495, 1595, 1599, 1615, 1679, 1703, 1739, 1771, 1855, 1887, 1991,  \\ & \\ 
3 $\circlegraph $  & \scriptsize 30, 66, 107, 138, 174, 255, 282, 302, 318, 354, 419, 498, 503, 759, 771, 795, 835, 863, 1007, 1319, \\ & \scriptsize 1355, 1427, 1479, 1551, 1583, 1619, 1691, 1695, 1871, 1895, 1947, 1967,  \\ & \\ 
4 $\circlegraph $  & \scriptsize 33, 105, 114, 146, 177, 249, 258, 285, 290, 299, 321, 330, 341, 357, 374, 385, 393, 402, 413, \\ & \scriptsize 429, 465, 482, 595, 663, 915, 987, 1023, 1067, 1239, 1435, 1727, 1743, 1955, 1995,  \\ & \\ 
5 $\circlegraph $  & \scriptsize 1043, 1203, 1451,  \\ & \\ 
6 $\circlegraph $  & \scriptsize 102, 165, 246, 362, 390, 435, 1335, 1419, 1547,  \\ & \\ 
7 $\circlegraph $  & \scriptsize 587, 971,  \\ & \\ 
\hline
  \end{tabular}
}

\begin{tabular}{|l|l|}
\hline & \\
 \scriptsize  $\begin{array}{c}3{\rm -conjugacy} \\ {\rm classes  \medspace graph}\end{array}$ & \normalsize $m {\rm \medspace specifying \medspace \medspace Bianchi \medspace groups} \medspace \PSLO \medspace {\rm with \medspace this} \medspace 3{\rm -conjugacy \medspace classes  \medspace graph}$ \\
\hline & \\
8 $\circlegraph $  & \scriptsize 438, 1131, 1635,  \\ & \\ 
 $\edgegraph$  & \scriptsize 7, 19, 31, 43, 67, 79, 103, 127, 139, 151, 163, 199, 211, 223, 271, 283, 307, 379, 439, \\ & \scriptsize 463, 487, 499, 523, 571, 607, 619, 631, 691, 727, 739, 751, 787, 811, 823, 859, 883, 907, 919, 967, \\ & \scriptsize 991, 1039, 1051, 1063, 1123, 1171, 1231, 1279, 1303, 1399, 1423, 1447, 1459, 1471, 1483, 1531, \\ & \scriptsize 1543, 1567, 1579, 1627,  1663, 1699, 1723, 1759, 1783, 1831, 1867, 1987, 1999,  \\ & \\ 
 $\edgegraph$  $\coprod $  $\circlegraph $  & \scriptsize 39, 111, 183, 219, 291, 327, 331, 367, 471, 543, 579, 643, 723, 831, 939, 1011, 1047, 1087, 1119, \\ & \scriptsize 1191, 1227, 1263, 1291, 1299, 1327, 1371, 1623, 1803, 1839, 1879, 1951, 1983,  \\ & \\ 
 $\edgegraph$  $\coprod $ 2 $\circlegraph $  & \scriptsize 547, 1747,  \\ & \\ 
 $\edgegraph$  $\coprod $ 4 $\circlegraph $  & \scriptsize 687,  \\ & \\ 
 $\edgegraph$  $\coprod $ 10 $\circlegraph $  & \scriptsize 1731,  \\ & \\ 
2 $\edgegraph$  & \scriptsize 13, 37, 61, 91, 109, 157, 181, 229, 247, 277, 349, 373, 403, 421, 427, 511, 679, 703, 871, \\ & \scriptsize 1099, 1147, 1267, 1591, 1603, 1687, 1891, 1963,  \\ & \\ 
2 $\edgegraph$  $\coprod $  $\circlegraph $  & \scriptsize 73, 97, 193, 241, 259, 313, 337, 409, 457, 559, 763, 1651, 1939,  \\ & \\ 
2 $\edgegraph$  $\coprod $ 2 $\circlegraph $  & \scriptsize 21, 57, 93, 129, 201, 309, 381, 397, 399, 417, 453, 489, 651, 903, 1443, 1659, 1767, 1843,  \\ & \\ 
2 $\edgegraph$  $\coprod $ 3 $\circlegraph $  & \scriptsize 433, 1027, 1387,  \\ & \\ 
2 $\edgegraph$  $\coprod $ 8 $\circlegraph $  & \scriptsize 237,  \\ & \\ 
4 $\edgegraph$  & \scriptsize 217, 301, 469,  \\ & \\ 
4 $\edgegraph$  $\coprod $ 2 $\circlegraph $  & \scriptsize 133.  \\ & \\ 
\hline
  \end{tabular}

\vbox{
\subsection{Numbers of conjugacy classes in $2$--torsion} \label{Numerical evaluation of Kraemer's formulae in 2-torsion}

Denote by $\delta$ the number of finite ramification places of $\rationals(\sqrt{-m}\thinspace)$ over $\rationals$.
Let $k_+$ be the totally real number field $\rationals(\sqrt{m}\thinspace)$ and denote its ideal class number by $h_{k_+}$.
For $m \neq 1$, Kr\"amer introduces the following indicators:

$ z := \begin{cases} 2, & {\rm if} \medspace 2 \medspace {\rm is} \medspace {\rm the} \medspace {\rm norm} \medspace {\rm of} \medspace {\rm an} \medspace {\rm integer} \medspace {\rm of} \medspace k_+, \\
1,  & {\rm otherwise,} \end{cases}
\hfill
 q := \begin{cases} 2, & {\rm if} \medspace \pm 2 \medspace {\rm is} \medspace {\rm the} \medspace {\rm norm} \medspace {\rm of} \medspace {\rm an} \medspace {\rm integer} \medspace {\rm of} \medspace k_+, \\
1,  & {\rm otherwise,} \end{cases}$
$$ w := \begin{cases} 2, & {\rm if} \medspace \forall \medspace {\rm prime} \medspace {\rm divisors} \medspace p \medspace {\rm of} \medspace m \medspace{\rm with} \medspace p \neq 2 \medspace {\rm we} \medspace {\rm have} \medspace p \equiv \pm 1 \mod 8, \\
1,  &  {\rm if} \medspace m \medspace {\rm admits} \medspace {\rm prime} \medspace {\rm divisors}  \medspace p \equiv \pm 3 \mod 8. \end{cases}$$
Furthermore, denote by $\epsilon := \frac{1}{2}(a+b\sqrt{m}) > 1$ the fundamental unit of~$k_+$ (where $a, b \in \N$).
Now, define
\\
$ x := \begin{cases} 2, & {\rm if} \medspace {\rm the} \medspace {\rm norm} \medspace {\rm of} \medspace \epsilon \medspace {\rm is} \medspace 1, \\
1,  & {\rm if} \medspace {\rm the} \medspace {\rm norm} \medspace {\rm of} \medspace \epsilon \medspace {\rm is} \medspace -1 \end{cases}$
\hfill
 and
\hfill
$ y := \begin{cases} 3, & {\rm if} \medspace b \equiv 0 \mod 2, \\
1,  & {\rm if} \medspace b \equiv 1 \mod 2. \end{cases}$
\\
Then \cite{Kraemer}*{26.12 and 26.14} yield the following formulae in $2$--torsion.
$$\begin{array}{|l|c|c|c|c|c|}
\hline
 m {\rm \medspace specifying \medspace Bianchi \medspace groups } \medspace \PSLO & \mu_T & \mu_2^- & \lambda_4^T & \lambda_4^* & \lambda_4 -\lambda_4^* \\
\hline & &&&&\\
m \equiv 7 \mod 8 & 0 & 0 & 0 & 0 & \frac{z}{2}h_{k_+} \\
& && &&\\
\hline & && &&\\
m \equiv 3 \mod 8 {\rm \medspace gives \medspace either } & 2^\delta & 0 & 2^{\delta -1} & 2^{\delta -1} & \frac{1}{2}(h_{k_+} -2^{\delta -1}) \\
& && &&\\
\medspace {\rm or} \medspace ({\rm provided} \medspace {\rm that} \medspace 2^{\delta -1} > 1) & 0 & 0 & 0 & 0 & \frac{1}{2}h_{k_+} \\
& && &&\\
\hline & && &&\\
m \equiv 2 \mod 4 \medspace {\rm and} \medspace w = 2 {\rm \medspace gives \medspace either } & 2^{\delta -1} & 2^{\delta -1} & 2^{\delta -2}z & 2^{\delta} & \frac{1}{4}x(z+2)h_{k_+} -2^{\delta -1} \\
& && &&\\
\medspace {\rm or} \medspace ({\rm provided} \medspace {\rm that} \medspace 2^{\delta -1} > 1)& 0 & 0 & 0 & 0 & \frac{1}{4}x(z+2)h_{k_+} \\
& && &&\\
\hline & && &&\\
m \equiv 2 \mod 4 \medspace {\rm and} \medspace w = 1 {\rm \medspace gives \medspace either } & 2^{\delta -1} & 0 & 2^{\delta -2} & 2^{\delta-2} & \frac{1}{2}(\frac{3}{2}xh_{k_+} -2^{\delta -2}) \\
& && &&\\
 \medspace {\rm or}  & 0 & 2^{\delta -1} & 0 & 2^{\delta -2}3 & \frac{3}{2}(\frac{1}{2}xh_{k_+} -2^{\delta -2}) \\
& && &&\\
\medspace {\rm or} \medspace ({\rm provided} \medspace {\rm that} \medspace 2^{\delta -1} > 2) & 0 & 0 & 0 & 0 & \frac{3}{4}x h_{k_+} \\
& && &&\\
\hline & && &&\\
m \equiv 1 \mod 8 \medspace {\rm and} \medspace m \neq 1 \medspace {\rm and} \medspace w = 2 {\rm \medspace gives \medspace either } & 2^{\delta -1} & 2^{\delta -1} & 2^{\delta -2} & 2^{\delta} & 2x h_{k_+} -2^{\delta -1} \\
& && &&\\
\medspace {\rm or} \medspace ({\rm provided} \medspace {\rm that} \medspace 2^{\delta -2} > 1)& 0 & 0 & 0 & 0 & 2x h_{k_+} \\
& && &&\\
\hline & && &&\\
m \equiv 1 \mod 8 \medspace \medspace {\rm and} \medspace w = 1 {\rm \medspace gives \medspace either } & 2^{\delta -1} & 0 & 2^{\delta -2} & 2^{\delta-2} & 2x h_{k_+} -2^{\delta -3} \\
& && &&\\
 \medspace {\rm or}  & 0 & 2^{\delta -1} & 0 & 2^{\delta -2}3 & 2x h_{k_+} -2^{\delta -3}3 \\
& && &&\\
\medspace {\rm or} \medspace ({\rm provided} \medspace {\rm that} \medspace 2^{\delta -2} > 2) & 0 & 0 & 0 & 0 & 2x h_{k_+} \\
& && &&\\
\hline & && &&\\
m \equiv 5 \mod 8 & 0 & 2^{\delta -1} & 0 & 2^{\delta -2}3 & \frac{1}{2}\left(x(2y+1) h_{k_+} -2^{\delta -2}3 \right) \\
& && &&\\
\medspace {\rm or} \medspace ({\rm provided} \medspace {\rm that} \medspace 2^{\delta -2} > 1) & 0 & 0 & 0 & 0 & \frac{1}{2}x (2y+1) h_{k_+} \\
\hline
\end{array}$$
}

\vbox{
The above case distinctions come from the fact that Kr\"amer's theorem 26.12 ranges over all types of maximal orders in quaternion algebras over $\rationals(\sqrt{-m}\thinspace)$, in which Kr\"amer determines the numbers of conjugacy classes in the norm-1-group.
The remaining task in order to decide which of the cases applies, is to find out of which type considered in the mentioned theorem is the maximal order M$_2(\ringO_{-m})$.
Some methods to cope with this task are introduced in \cite{Kraemer}*{\S 27},
 where Kr\"amer obtains the following criteria for the $2$--torsion numbers:
$$\begin{array}{|l|l|}
\hline & \\
   {\rm condition} & {\rm implication} \\
\hline & \\
m \equiv 7 \mod 8 & \mu_T = \mu_2^- = \lambda_4^T = \lambda_4^* = 0. \\ & \\
m \equiv 5 \mod 8 & \mu_T = \lambda_4^T = 0. \\& \\
m \equiv 21 \mod 24 & \lambda^*_4 = 0. \\& \\
m \equiv 0 \mod 6 \medspace {\rm and} \medspace \lambda^*_4 > 0 & \lambda_4^T > 0. \\& \\
m \equiv 9 \mod 24 \medspace {\rm and} \medspace \lambda^*_4 > 0 & \lambda_4^T > 0. \\& \\
m \medspace {\rm prime} \medspace {\rm and} \medspace m \equiv 1 \medspace {\rm or} \medspace 3 \mod 8 & \lambda_4^T > 0. \\& \\
m \equiv 5 \mod 8 \medspace {\rm and} \medspace m  \medspace {\rm prime}& \lambda^*_4 > 0. \\&\\
m = 2p \medspace {\rm with} \medspace p \medspace {\rm prime} \medspace {\rm and} \medspace p \equiv 3 \medspace {\rm or} \medspace 5 \mod 8 & \lambda^*_4 > 0. \\& \\
m = p'p \medspace {\rm with} \medspace p  \medspace {\rm and} \medspace p' \medspace {\rm prime} \medspace {\rm and} \medspace p \equiv p' \equiv 3 \medspace {\rm or} \medspace 5 \mod 8 & \lambda^*_4 > 0. \\& \\
m = 3p \medspace {\rm with} \medspace p \medspace {\rm prime} \medspace {\rm and} \medspace p \equiv 1 \medspace {\rm or} \medspace 3 \mod 8 & \lambda^T_4 > 0. \\& \\

m \equiv 1 \medspace {\rm or} \medspace 2 \mod 4 \medspace {\rm and} \medspace m \neq 1 \medspace {\rm and} \medspace x = 1 & \lambda_4^* > 0 \medspace {\rm and} \medspace \mu_2^- > 0. \\& \\

m \equiv 1 \medspace {\rm or} \medspace 2 \mod 4 \medspace {\rm and} \medspace m \neq 1 \medspace {\rm and} \medspace x = 2 & \lambda_4 -\lambda_4^* > 0. \\& \\
m \equiv 3  \mod 8 \medspace {\rm and} \medspace -2 \medspace {\rm occurs} \medspace {\rm as} \medspace {\rm norm}  \medspace {\rm on} \medspace \ringO_{k_+} & \lambda_4^* > 0 \medspace {\rm and} \medspace \lambda_4^T > 0. \\& \\

m \equiv 3  \mod 8 \medspace {\rm and} \medspace -2 \medspace {\rm does}  \medspace {\rm not} \medspace {\rm occur} \medspace {\rm as} \medspace {\rm norm}  \medspace {\rm on} \medspace \ringO_{k_+} & \lambda_4 -\lambda_4^* > 0. \\& \\

m \equiv 3  \mod 8 \medspace {\rm and} \medspace m \medspace {\rm admits} \medspace {\rm a} \medspace {\rm prime}  \medspace {\rm divisor} \medspace p \medspace {\rm with} \medspace p \equiv 5 \medspace {\rm or} \medspace 7 \mod 8  & \lambda_4 -\lambda^*_4 > 0. \\& \\

m \equiv 1  \mod 8 \medspace {\rm and} \medspace w = 1 \medspace {\rm and} \medspace h(k_+) = 2^{\delta -3} & \mu^-_2 = 0. \\& \\
m \equiv 2  \mod 4 \medspace {\rm and} \medspace -2 \medspace {\rm occurs} \medspace {\rm as} \medspace {\rm norm}  \medspace {\rm on} \medspace \ringO_{k_+} & \lambda_4^T > 0. \\& \\

m \equiv 2  \mod 4 \medspace {\rm and} \medspace -2 \medspace {\rm does}  \medspace {\rm not} \medspace {\rm occur} \medspace {\rm as} \medspace {\rm norm}  \medspace {\rm on} \medspace \ringO_{k_+} \medspace {\rm and} \medspace h(k_+) = 2^{\delta -2}& \lambda_4^* = 0. \\& \\

m \equiv 2  \mod 4 \medspace {\rm and} \medspace q = 1 \medspace {\rm and} \medspace h(k_+) = 2^{ \delta -1} \medspace {\rm and} \medspace w = 2 & \lambda^*_4 = 0. \\& \\

m \equiv 2  \mod 4 \medspace {\rm and} \medspace  h(k_+) = 2^{\delta-2}  &\\
\medspace {\rm and} \medspace m \medspace {\rm admits} \medspace {\rm a} \medspace {\rm prime}  \medspace {\rm divisor} \medspace p \medspace {\rm with} \medspace p \equiv 5 \medspace {\rm or} \medspace 7 \mod 8 & \lambda^*_4 = 0. \\& \\

\hline
  \end{array}$$
}

\vbox{
With the above criteria at hand, we can decide for many Bianchi groups, which of the alternative cases in Kr\"amer's formulae must be used.
We do this in the below tables for all such Bianchi groups $\PSLO$ with absolute value of the discriminant $\Delta$ ranging between $7$ and $2003$.
In the cases $m \in \{$ 34, 105, 141, 142, 194, 235, 323, 427, 899, 979, 1243, 1507$\}$, where these statements are not sufficient to eliminate the wrong alternatives, we insert the results of~\cite{BianchiGP}. This way, the below tables treat all Bianchi groups with units $\{\pm 1\}$ and discriminant of absolute value less than $820$. The cases where an ambiguity remains (so to exclude them from our tables) are the following values of $m$: 205, 221, 254, 273, 305, 321, 322, 326, 345, 377, 381, 385, 386, 410, 438, 465, 469, 473, 482, 1067, 1139, 1211, 1339, 1443, 1763, 1771, 1947.

The above mentioned theorem on subgroup occurrences \cite{manuscript} solves all these ambiguities.
\bigskip

\begin{tabular}{|l|l|}
\hline & \\
 \scriptsize  $\begin{array}{c}2{\rm -torsion} \\ {\rm homology}\end{array}$ & \normalsize $m {\rm \medspace specifying \medspace \medspace Bianchi \medspace groups} \medspace \PSLO \medspace {\rm with \medspace this} \medspace 2{\rm -torsion \medspace homology}$ \\
\hline & \\ 
\normalsize \normalsize $P_{\circlegraph}$  & \scriptsize 7, 15, 23, 31, 35, 39, 47, 55, 71, 87, 91, 95, 103, 111, 115, 127, 143, 151, 155,  \\ & \scriptsize
159, 167, 183, 191, 199, 203, 215, 239, 247, 259, 263, 271, 295, 299, 303, 311, 319, 327, 335,  \\ & \scriptsize
355, 367, 371, 383, 395, 403, 407, 415, 431, 447, 463, 471, 479, 487, 503, 515, 519, 535, 543,  \\ & \scriptsize
551, 559, 583, 591, 599, 607, 611, 631, 635, 647, 655, 667, 671, 687, 695, 703, 707, 719, 743,  \\ & \scriptsize
751, 755, 763, 767, 807, 815, 823, 831, 835, 851, 863, 871, 879, 887, 911, 919, 923, 951, 955,  \\ & \scriptsize
967, 983, 991, 995, 1007, 1027, 1031, 1039, 1043, 1047, 1055, 1063, 1079, 1099, 1103, 1115, \\ & \scriptsize 
1119, 1135, 1147, 1151, 1159, 1167, 1195, 1199, 1219, 1231, 1247, 1255, 1263, 1267, 1279, 
 \\ & \scriptsize
1303, 1315, 1319, 1355, 1363, 1379, 1383, 1391, 1399, 1403, 1415, 1423, 1439, 1447, 1471, \\ & \scriptsize
1487, 1511, 1535, 1543, 1555, 1559, 1583, 1591, 1603, 1607, 1623,  1643, 1651, 1655, 1663,
\\ & \scriptsize
 1671, 1703, 1711, 1727, 1739, 1759, 1783, 1795, 1807, 1823, 1831, 1835, 1839, 1871, 1879,  \\ & \scriptsize
1883, 1891, 1895, 1903, 1915, 1919, 1939, 1943, 1951, 1959, 1963, 1983, 1991, 1999,  \\ & \\
\normalsize 2\normalsize $P_{\circlegraph}$  & \scriptsize 14, 46, 62, 94, 119, 158, 195, 206, 231, 255, 287, 302, 334, 382, 391, 398, 435, 446, 455,  \\ & \scriptsize
478, 483, 511, 527, 555, 595, 615, 623, 651, 663, 679, 715, 759, 791, 795, 903, 915, 935, 943,  \\ & \scriptsize
987, 1015, 1095, 1131, 1207, 1235, 1271, 1295, 1311, 1335, 1343, 1407, 1435, 1455, 1463, \\ & \scriptsize
1479, 1491, 1515, 1547,  1551, 1595, 1615, 1631, 1635, 1659, 1687, 1695, 1751, 1767, 1799,
\\ & \scriptsize 1855, 1887, 1927, 1955, 1967,  \\ & \\
\normalsize 3\normalsize $P_{\circlegraph}$  & \scriptsize 21, 30, 42, 69, 70, 77, 78, 79, 93, 110, 133, 138, 154, 174, 182, 186, 190, 213, 222,   
223, \\ & \scriptsize 230, 235, 237, 253, 266, 282, 286, 301, 309, 310, 318, 341, 359, 366, 406, 413, 426, 427,  \\ & \scriptsize 430, 437, 
453, 470, 474, 494, 839, 895, 899, 1191, 1223, 1367, 1527, 1567, 1639, 1735, 1847,  \\ & \\
\normalsize 4\normalsize $P_{\circlegraph}$  & \scriptsize 161, 217, 238, 329, 399, 497, 799, 959, 1023, 1155, 1239, 1351, 1679, 1743, 1995,  \\ & \\
\normalsize 5\normalsize $P_{\circlegraph}$  & \scriptsize 439, 727, 1111, 1327,  \\ & \\
\normalsize 6\normalsize $P_{\circlegraph}$  & \scriptsize  142, 165, 210, 285, 330, 357, 390, 429, 434, 462, 1495, 1599,  \\ & \\
\normalsize 7\normalsize $P_{\circlegraph}$  & \scriptsize 141, 1087,  \\ & \\
\normalsize 8\normalsize $P_{\circlegraph}$  & \scriptsize 105,  \\ & \\
\normalsize 2$P_{\Kleinfourgroup}^*$   & \scriptsize 5, 10, 13, 26, 29, 53, 58, 61, 74, 106, 109, 122, 149, 157, 173, 181, 202, 218, 277,  \\ & \scriptsize
293, 298, 314, 317, 362, 394, 397, 421, 458, 461,  \\ & \\
\hline
\end{tabular}
}

\begin{tabular}{|l|l|}
\hline & \\
 \scriptsize  $\begin{array}{c}2{\rm -torsion} \\ {\rm homology}\end{array}$ & \normalsize $m {\rm \medspace specifying \medspace \medspace Bianchi \medspace groups} \medspace \PSLO \medspace {\rm with \medspace this} \medspace 2{\rm -torsion \medspace homology}$ \\
\hline & \\ 
\normalsize 2$P_{\Kleinfourgroup}^*$ + 2\normalsize $P_{\circlegraph}$  & \scriptsize 37, 101, 197, 269, 349, 373, 389,  \\ & \\
\normalsize 2$P_{\Kleinfourgroup}^*$ + 3\normalsize $P_{\circlegraph}$  & \scriptsize 229, 346,  \\ & \\
\normalsize 4$P_{\Kleinfourgroup}^*$   & \scriptsize 85, 130, 170, 290, 365, 370, 493,  \\ & \\
\normalsize 4$P_{\Kleinfourgroup}^*$ + \normalsize $P_{\circlegraph}$  & \scriptsize 65, 185, 265, 481,  \\ & \\
\normalsize 4$P_{\Kleinfourgroup}^*$ + 3\normalsize $P_{\circlegraph}$  & \scriptsize 442, 445,  \\ & \\
\normalsize 4$P_{\Kleinfourgroup}^*$ + 4\normalsize $P_{\circlegraph}$  & \scriptsize 485,  \\ & \\
\normalsize 4$P_{\Kleinfourgroup}^*$ + 5\normalsize $P_{\circlegraph}$  & \scriptsize 145,  \\ & \\
\normalsize $P_{\Afour}^*$ + $P_{\Kleinfourgroup}^*$   & \scriptsize 2,  \\ & \\
\normalsize 2$P_{\Afour}^*$  & \scriptsize 11, 19, 43, 59, 67, 83, 107, 131, 139, 163, 179, 211, 227, 251, 283, 307, 331, 347, 379,  \\ & \scriptsize
419, 467, 491, 523, 547, 563, 571, 587, 619, 643, 683, 691, 739, 787, 811, 827, 859, 883, 907,  \\ & \scriptsize
947, 971, 1019, 1051, 1123, 1163, 1187, 1259, 1283, 1291, 1307, 1427, 1451, 1459, 1483, 1499,  \\ & \scriptsize 1531, 1571, 1579, 
1619, 1667, 1699, 1723, 1747, 1867, 1931, 1979, 2003,  \\ & \\
\normalsize 2$P_{\Afour}^*$ + \normalsize $P_{\circlegraph}$  & \scriptsize 6, 22, 38, 86, 118, 134, 166, 214, 262, 278, 358, 422, 443, 454, 659, 1091, 1171, 1523, 1627,  \\ & \scriptsize
1787, 1811, 1907, 1987,  \\ & \\
\normalsize 2$P_{\Afour}^*$ + 2\normalsize $P_{\circlegraph}$  & \scriptsize 499,  \\ & \\
\normalsize 2$P_{\Afour}^*$ + 2$P_{\Kleinfourgroup}^*$   & \scriptsize 17, 41, 73, 89, 97, 113, 137, 193, 233, 241, 281, 313, 337, 353, 409, 433, 449, 457,  \\ & \\
\normalsize 2$P_{\Afour}^*$ + 2$P_{\Kleinfourgroup}^*$ + \normalsize $P_{\circlegraph}$  & \scriptsize 82, 146, 178, 274, 466,  \\ & \\
\normalsize 2$P_{\Afour}^*$ + 2$P_{\Kleinfourgroup}^*$ + \normalsize 2$P_{\circlegraph}$  & \scriptsize 34, 194,  \\ & \\
\normalsize 2$P_{\Afour}^*$ + 2$P_{\Kleinfourgroup}^*$ + 4\normalsize $P_{\circlegraph}$  & \scriptsize 226, 257,  \\ & \\
\normalsize 2$P_{\Afour}^*$ + 2$P_{\Kleinfourgroup}^*$ + 8\normalsize $P_{\circlegraph}$  & \scriptsize 401,  \\ & \\
\normalsize 4$P_{\Afour}^*$   & \scriptsize 51, 123, 187, 267, 339, 411, 451, 699, 771, 779, 803, 843, 1059, 1203, 1347, \\ & \scriptsize 1563, 1691, 1707, 1779, 
1819, 1843, 1923,  \\ & \\
\normalsize 4$P_{\Afour}^*$ + \normalsize $P_{\circlegraph}$  & \scriptsize 219, 291, 323, 579, 723, 731, 939, 979, 1003, 1011, 1227, 1243, 1371, 1387, 1411, 1507, 1731, \\ & \scriptsize 1803,  \\ & \\
\normalsize 4$P_{\Afour}^*$ + 2\normalsize $P_{\circlegraph}$  & \scriptsize 66, 102, 114, 246, 258, 354, 374, 402, 418, 498, 1851,  \\ & \\
\normalsize 4$P_{\Afour}^*$ + 3\normalsize $P_{\circlegraph}$  & \scriptsize 33, 57, 129, 177, 201, 209, 249, 393, 417, 489, 1299,  \\ & \\
\normalsize 8$P_{\Afour}^*$   & \scriptsize 627, 1419.  \\ & \\
\hline
\end{tabular}

\subsection{Asymptotic behavior of the number of conjugacy classes} \label{Asymptotic behaviour of the number of conjugacy classes}
\begin{figure}
\caption{Average homological 3-torsion outside subgroups of type $\Sthree$, scaled as indicated.} \label{plot}
\includegraphics[height=100mm]{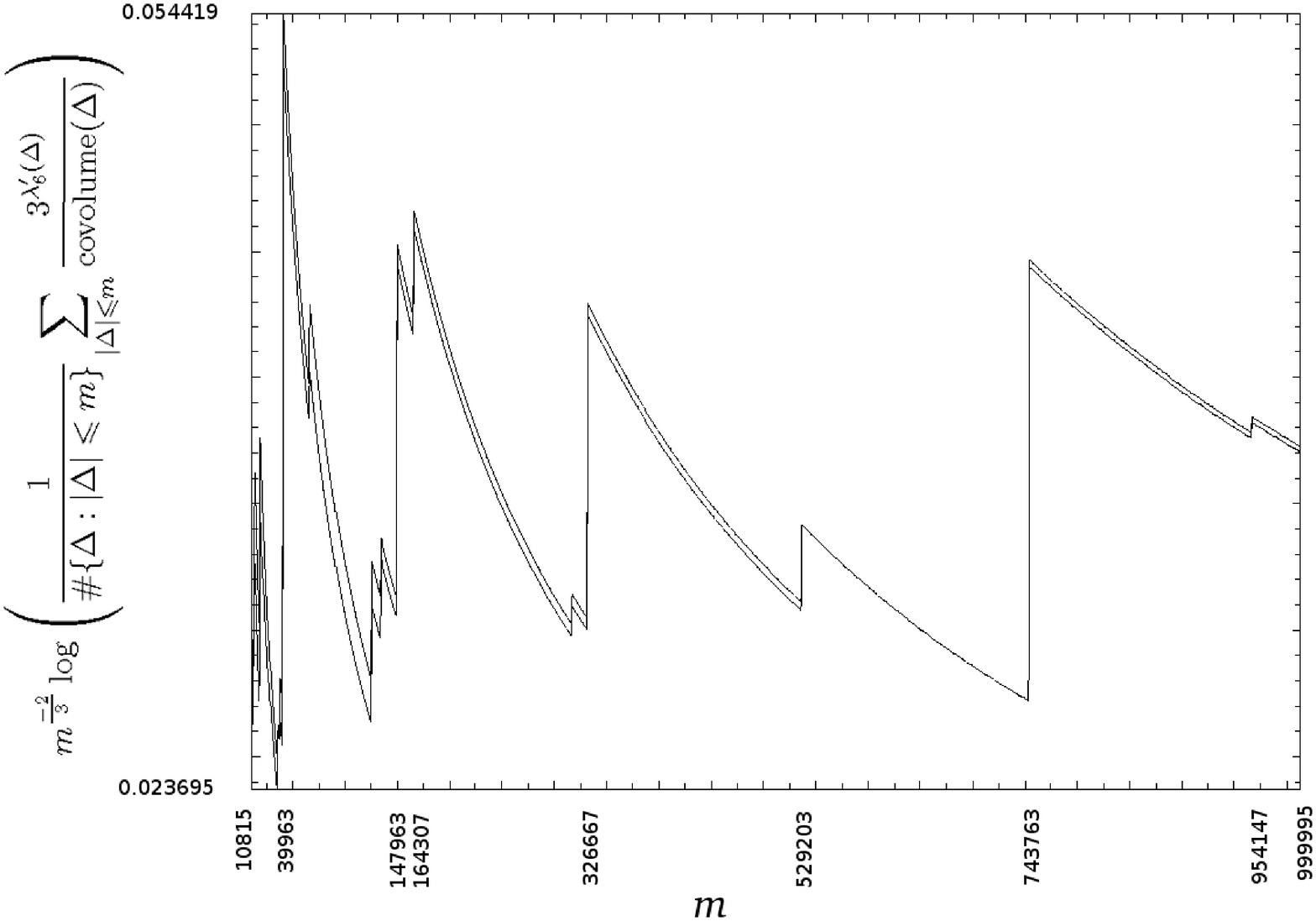}
\end{figure}
From Kr\"amer's above formulae, we see that both in $2$--torsion and in $3$--torsion, the number of conjugacy classes of finite subgroups, and hence also the cardinality of the homology of the Bianchi groups in degrees above their virtual cohomological dimension, admits only two factors which are not strictly limited: $h_{k_+}$ and $2^\delta$.
As for the ideal class number $h_{k_+}$, it is subject to the predictions of the Cohen-Lenstra heuristic \cite{Cohen}. 
As for the factor $2^\delta$, the number $\delta$ of finite ramification places of $\rationals(\sqrt{-m}\thinspace)$ over $\rationals$ is well-known to equal the number of prime divisors of the discriminant of $\rationals(\sqrt{-m}\thinspace)$.

The numerical evaluation of Kr\"amer's formulae provides us with databases which are over a thousand times larger than what is reasonable
 to print in Sections~\ref{Numerical evaluation of Kraemer's formulae in 3-torsion} and~\ref{Numerical evaluation of Kraemer's formulae in 2-torsion}.
We now give an instance of how these databases can be exploited.
Denote the discriminant of $\rationals(\sqrt{-m}\thinspace)$ by~$\Delta$. In the cases $m \equiv 3 \mod 4$, we have~$\Delta = -m$.
Denote the number $\lambda_6 -\lambda_6^*$ of connected components of type $\circlegraph$ in the 3-conjugacy classes graph by ${\lambda'_6(\Delta)}$.
Then clearly, the subgroup in $\Homol_{q}(\PSLO)$, $q > 2$, generated by the order-3-elements coming from the connected components of this type, is of order $3^{\lambda'_6(\Delta)}$.
Denote by ${\rm covolume}(\Delta)$ the volume of the quotient space $_{\PSLO}\backslash\Hy$.
The study of the ratio $\frac{3^{\lambda'_6(\Delta)}}{{\rm covolume}(\Delta)}$ is motivated by the formulae in~\cite{BergeronVenkatesh}.
In Figure~\ref{plot}, we print the logarithm of the average of this ratio over the cases $|\Delta| \equiv 3 \mod 4$, scaled by a factor~$m^\frac{-2}{3}$, so to say
$$m^\frac{-2}{3}\log\left(\frac{1}{\# \{\Delta : |\Delta| \leq m\}}\sum\limits_{|\Delta| \leq m} \frac{3^{\lambda'_6(\Delta)}}{{\rm covolume}(\Delta)}\right),$$
where we consider $m$ and $\Delta$ as independent variables, $m$ running through the square-free positive rational integers.
In order to cope with the fact that in some cases, Kr\"amer's formulae leave an ambiguity, we print a function assuming the lowest possible values of ${\lambda'_6(\Delta)}$ and one assuming the highest possible values of~${\lambda'_6(\Delta)}$ in the same diagram.

So for $m$ greater than 10815 and less than one million, we can observe that the average of the above ratio oscillates between exp($m^\frac{2}{3}0.023695$) and exp($m^\frac{2}{3}0.054419$).
For $m$ less than 10815, this oscillation is much stronger, and the diagram might be seen as suggesting that possibly the oscillation could remain between these two bounds for $m$ greater than one million.

For related asymptotics, see the recent works of Bergeron/Venkatesh~\cite{BergeronVenkatesh} and Seng\"un~\cite{Sengun}. 
For an alternative computer program treating the Bianchi groups, see the SAGE package of Cremona's student Aran\'es~\cite{Aranes}, and for GL$_2(\ringO)$ see Yasaki's program \cite{Yasaki}.
\end{appendix}

\end{document}